\documentclass[11pt,a4paper]{article}
\pdfoutput=1 
\usepackage[margin=1in]{geometry}
\usepackage{amsmath,amsthm,amssymb,enumitem, bbm ,graphicx,color,caption,upgreek, float, tikz, subcaption,booktabs,longtable, appendix,graphics, pdfpages,rotating,mathtools, array, bm, blkarray,setspace,textcomp,subcaption}
\definecolor{firebrick}{rgb}{0.7, 0.13, 0.13}

\usepackage[pdftex]{hyperref}

\setlist{noitemsep, topsep=5pt,leftmargin=*}

\definecolor{blauw}{RGB}{61,158,255}
\definecolor{donkerblauw}{RGB}{0,0,255}
\definecolor{donkergroen}{RGB}{46,148,0}
\definecolor{donkerrood}{RGB}{204,0,0}

\makeatletter 
\newcommand\mynobreakpar{\par\nobreak\@afterheading} 
\makeatother

\makeatletter
\let\@fnsymbol\@arabic
\makeatother

\newcommand{\N}{\mathbb{N}}
\newcommand{\Z}{\mathbb{Z}}
\newcommand{\C}{\mathbb{C}}
\newcommand{\R}{\mathbb{R}}

\let\OLDthebibliography\thebibliography
\renewcommand\thebibliography[1]{
  \OLDthebibliography{#1}
  \setlength{\parskip}{0pt}
  \setlength{\itemsep}{0pt plus 0.3ex}
}

\usepackage[english]{babel}

\newtheorem{theorem}{Theorem}[section]
\newtheorem{lemma}[theorem]{Lemma}

\newtheorem{proposition}[theorem]{Proposition}
\newtheorem{corollary}[theorem]{Corollary}

\theoremstyle{definition}
\newtheorem{examp}[theorem]{Example}
\newtheorem{defn}{Definition}[section]

\newtheorem*{examp*}{Example}

\newtheorem{procedure}{Procedure}[section]

\newtheorem{remark}{Remark}[section]

\theoremstyle{plain}
\newfloat{Algorithm}{!hbt}{alg}
\newcommand{\T}{^{\sf T}}

\newcounter{thm}[section]

\title{Symmetry reduction to optimize a graph-based polynomial from queueing theory} \date{December 31, 2021}
\author{Sven Polak\thanks{Centrum Wiskunde \& Informatica (CWI), Amsterdam. E-mail: \href{mailto:s.c.polak@cwi.nl}{\texttt{s.c.polak@cwi.nl}}.}}
\begin{document}
\maketitle
\setcounter{footnote}{1}

\noindent \textbf{Abstract.}  For given integers $n$ and $d$, both at least 2, we consider a homogeneous multivariate polynomial $f_d$ of degree $d$ in variables indexed by the edges of the complete graph on $n$ vertices and coefficients depending on cardinalities of certain unions of edges.  Cardinaels, Borst and Van Leeuwaarden  (arXiv:2111.05777, 2021) asked whether $f_d$, which arises in a model of job-occupancy in redundancy scheduling,  attains its minimum over the standard simplex at the uniform  probability vector. Brosch, Laurent and Steenkamp [SIAM J.\ Optim.\ 31 (2021), 2227--2254] proved that $f_d$ is convex over the standard simplex if $d=2$ and $d=3$, implying the desired result for these $d$. 

We give a symmetry reduction to show that for fixed $d$, the  polynomial is convex over the standard simplex (for all $n\geq 2$) if a constant number of constant matrices (with size \emph{and coefficients} independent of $n$) are positive  semidefinite. This result is then used in combination with a  computer-assisted verification to show that the polynomial~$f_d$ is convex  for $d\leq 9$.

\,$\phantom{0}$


\section{Introduction}
This paper is inspired by a recent paper of Brosch, Laurent and Steen-kamp \cite{fdpoly} and partially answers a question originating from queueing theory asked by Cardinaels, Borst and Van Leeuwaarden~\cite{fdpolyoriginal}. The latter authors asked whether a certain homogeneous  multivariate polynomial of degree~$d$, in variables indexed by the edges of the complete graph on~$n$ vertices, attains its minimum over the standard simplex at the uniform probability vector.  Brosch, Laurent and Steenkamp~\cite{fdpoly} showed that the polynomial is convex over the standard simplex if $d=2$ and $d=3$ (using symmetry properties of the polynomial in combination with results about the Hamming and Johnson schemes), and that this implies the desired result for these~$d$. 

The main contribution of the present paper is a symmetry reduction which reformulates the problem of checking whether~$f_d$ is convex for all~$n$ as a problem depending only on~$d$, fully independent of~$n$. This helps to prove the conjecture for~$d \leq 9$ and is a promising starting point for further work on the conjecture.

Given integers~$n,L \geq 2$, set~$V:=[n] = \{1,\ldots,n\}$ and~$E:=\{e \subseteq V \, : \, |e| =L \}$, so that~$(V,E)$ is the complete $L$-uniform hypergraph on~$n$ elements. Set~$m:=|E|=\tbinom{n}{L}$.
Given an integer~$d \geq 2$, consider the following~$m$-variate polynomial in variables~$x=(x_e \, : \, e \in E)$:
\begin{align} \label{fdeq} 
f_d(x) = \sum_{(e_1,\ldots,e_d) \in E^d}\prod_{i=1}^d \frac{x_{e_i}}{|e_1 \cup \ldots \cup e_i|},
\end{align}
which is homogeneous of degree~$d$. Denote the standard simplex in~$\R^m$ by
$$
\Delta_m := \left\{  x = (x_e)_{e \in E} \in \R^m \,\,:\,\, x \geq 0, \,\, \textstyle\sum_{e \in E} x_e =1  \right\}.
$$
Cardinaels, Borst and Van Leeuwaarden~\cite{fdpolyoriginal} asked whether~$f_d$ attains its minimum over the standard simplex~$\Delta_m$ at the uniform probability vector~$x^* = \tfrac{1}{m}(1,\ldots,1)$, specifically for the case~$L=2$.

Brosch, Laurent and Steenkamp~\cite{fdpoly} showed that~$f_d$ is convex over~$\Delta_m$ if~$d=2$, and also if~$L=2$ and~$d=3$, and they gave numerical evidence for the validity of the claim that~$f_d$ is convex over~$\Delta_m$ for various small values of~$n,d$ and~$L$. They observed that convexity of~$f_d$ implies that~$f_d$ attains its minimum over~$\Delta_m$ at the uniform probability vector~$x^* =\tfrac{1}{m}(1,\ldots,1)$, as~$f_d$ satisfies a certain invariance property under permutations of~$[n]$.

 In this paper we will use a symmetry reduction to reduce the problem to constant size with constant coefficients (independent of~$n$, for fixed~$d$) in combination with a computer-assisted verification to extend the result from~$\cite{fdpoly}$ to~$L=2$ and~$d \leq 9$.
\begin{theorem} \label{theorem2}
For~$d \leq 9$ and~$L=2$, the polynomial~$f_d$ from~\eqref{fdeq} is convex over the standard simplex~$\Delta_m$.  
\end{theorem}
\begin{corollary}
For~$d \leq 9$ and~$L=2$, the polynomial~$f_d$ from~\eqref{fdeq} attains its global minimum over the standard simplex~$\Delta_m$ at  the uniform probability vector~$x^* = \tfrac{1}{m}(1,\ldots,1)$. 
\end{corollary}

In order to prove Theorem~\ref{theorem2}, we first prove the following general new result, which might be applicable more widely. For~$n \in \N$, let~$S_n$ denote the symmetric group on~$n$ elements and set~$[n]:= \{1,\ldots,n\}$. Fix a nonnegative integer~$k$, and let~$S_{n-k}$ denote the subgroup of~$S_n$ consisting of all~$\sigma \in S_n$ with~$\sigma(i)=i$ for all~$i \in [k]$. Then~$S_{n-k}$ acts on~$[n]$, hence on~$\tbinom{[n]}{2}$ via~$\sigma \cdot \{ i,j\} = \{\sigma(i), \sigma(j)\}$ for~$\{i,j\} \in \tbinom{[n]}{2}$ and~$\sigma \in S_{n-k}$. 
\begin{theorem}\label{Antheorem}
Let~$k \geq 0$ be a fixed integer. Suppose that $(A^{(n)})_{n \geq k}$ is a sequence of symmetric matrices such that:
\begin{enumerate}[label=(\roman*)]
 \item\label{1th} $A^{(n)} \in \R^{\tbinom{[n]}{2} \times \tbinom{[n]}{2}}$ for each~$n \geq k$,
\item\label{2th} For all $n,n'\in \N$ with $k \leq n'\leq n$ and all~$e_i, e_j \in \tbinom{[n']}{2}$, it holds that~$A^{(n')}_{e_i,e_j} = A^{(n)}_{e_i,e_j}$,
\item\label{3th} For all~$n \geq k$, the matrix $A^{(n)}$ is invariant under the simultaneous action of~$S_{n-k}$ on its rows and columns.
\end{enumerate}
Then $A^{(n)}$ is positive semidefinite for every~$n \geq k$ if and only if
\begin{align} \label{blockvalue}
a_{(\{k+1,k+2\},\{k+1,k+2\})} - 2  a_{(\{k+1,k+2\},\{k+1,k+3\})} +  a_{(\{k+1,k+2\},\{k+3,k+4\})} \geq 0
\end{align}
 and the two matrices {\small
\begin{align*} 
\setlength\aboverulesep{1pt}\setlength\belowrulesep{1pt}
    \setlength\cmidrulewidth{0.5pt}
\begin{blockarray}{cccc}
 &  \tbinom{k}{2} &   k  & 1 \\
\begin{block}{c(c|c|c)}
   \tbinom{k}{2}&   \left( a_{(e_i,e_j)} \right)_{e_i,e_j \in \tbinom{[k]}{2}} &   \left( a_{(e_i,\{j,k+1\})} \right)_{\substack{e_i \in \tbinom{[k]}{2} \\ j\in [k]  }}   & \left( a_{(e_i,\{k+1,k+2\})} \right)_{\substack{e_i \in \tbinom{[k]}{2}}} \\
\cmidrule{2-4}
 k &  \left( a_{(e_i,\{j,k+1\})} \right)_{\substack{e_i \in \tbinom{[k]}{2} \\ j\in [k]  }}\T     &  (a_{(\{i,k+1\},\{j,k+2\})})_{i,j\in [k]}  &(a_{(\{i,k+1\},\{k+2,k+3\})})_{i \in [k]}  \\   \cmidrule{2-4}
  1&    \left( a_{(e_i,\{k+1,k+2\})} \right)_{\substack{e_i \in \tbinom{[k]}{2}}} \T  &  (a_{(\{i,k+1\},\{k+2,k+3\})})_{i \in [k]}\T   &  a_{(\{k+1,k+2\},\{k+3,k+4\})} \\
\end{block}
\end{blockarray}
\end{align*}}
and
{\footnotesize
\begin{align*} 
\setlength\aboverulesep{1pt}\setlength\belowrulesep{1pt}
    \setlength\cmidrulewidth{0.5pt}
\begin{blockarray}{ccc}
 &   k&   1 \\
\begin{block}{c(c|c)}
 k&    \left(a_{(\{i,k+1\},\{j,k+1\})}  -   a_{(\{i,k+1\},\{j,k+2\})}\right)_{i,j \in [k]} &    \left( a_{(\{i,k+1\},\{k+1,k+2\})} - a_{(\{i,k+1\},\{k+2,k+3\})}  \right)_{i \in [k]} \\\cmidrule{2-3}
  1& \left( a_{(\{i,k+1\},\{k+1,k+2\})} - a_{(\{i,k+1\},\{k+2,k+3\})}  \right)_{i \in [k]}\T  &      a_{(\{k+1,k+2\},\{k+1,k+3\})} - a_{(\{k+1,k+2\},\{k+3,k+4\})}  \\  
\end{block}
\end{blockarray} 
\end{align*}}
are positive semidefinite. Here we write~$a_{(e_i,e_j)}$ for the~$(e_i,e_j)$-th entry of~$A^{(k+4)}$, for~$e_i, e_j \in \tbinom{[m]}{2}$ where~$m \leq k+4$.
\end{theorem}
So if $(A^{(n)})_{n \geq k}$ is a sequence of matrices satisfying the theorem, then~$A^{(n)}$ is positive semidefinite for every~$n \geq k$ if and only if three matrices with size \emph{and coefficients} independent of~$n$ are positive semidefinite.
Note that~$A^{(n)}$ for~$n \geq k+4$ contains as many distinct values as~$A^{(k+4)}$, as each~$\omega \in(E \times E)/S_{n-k}$ contains a pair of edges~$(e,f) \in \omega$ with~$e,f \subseteq [k+4]$. This is why the constant matrices are constructed from~$A^{(k+4)}$.

The proof of Theorem~$\ref{Antheorem}$ (cf. Section~\ref{bigproof}) consists of the following ingredients.
\begin{enumerate}[label=(\alph*)]
\item  We use the fact that a matrix~$A^{(n)}$ satisfying conditions \ref{2th},\ref{3th} is a morphism of representations $\mathbb{R}^E \to \mathbb{R}^E$, and therefore orthogonally block-diagonalizable once a decomposition of $\mathbb{R}^E$ into $S_{n-k}$-irreducible submodules is available.
\item We compute an explicit decomposition of $\mathbb{R}^E$ into~$S_{n-k}$-irreducibles (this is Proposition~\ref{u1u2u3}). We do this from first principles and without relying on Young-tableaux. This results in a block-diagonalization of~$A^{(n)}$ into matrices of constant size \emph{but coefficients dependent on~$n$}.
\item We apply a sequence of combinatorial operations preserving positive semidefiniteness (such as dividing or multiplying rows and columns by the same constant) to the resulting three matrices to decompose them into a sum of matrices in a convenient way. Then we use a limit argument (letting~$n \to\infty$) so that we arrive at matrices with coefficients independent of~$n$. This part of the proof is crucial to eliminate the dependence on~$n$ in the resulting three matrices, and is the key to proving Theorem~\ref{Antheorem}.

This step is inspired by the necessary condition that~$A^{(n)} \succeq  0$ for \emph{large}~$n$, if $A^{(n)}\succeq 0$ for all~$n\geq k$. (Here we write~$A^{(n)}\succeq 0$ if~$A^{(n)}$ is positive semidefinite.) This observation results in a necessary condition that three constant `limit' matrices are positive semidefinite. It will turn out that positive semidefiniteness of these constant matrices is also sufficient to show~$A^{(n)} \succeq 0$  for all~$n \geq k$.
\end{enumerate}

\begin{remark} \label{k0remark}
Theorem~$\ref{Antheorem}$ in particular also holds for~$k=0$. Let~$(A^{(n)})_{n \geq 0}$ be a sequence of matrices satisfying the conditions of the theorem.
Then we have
\begin{align*}{\footnotesize
A^{(4)} = 
\begin{blockarray}{ccccccc} 
& \{1,2\} &  \{1,3\} &  \{1,4\} &  \{2,3\} &  \{2,4\} &  \{3,4\}  \\
\begin{block}{c(cccccc)}
  \{1,2\} &x & y & y & y  & y &z  \\
  \{1,3\} &y & x &y &  y & z &  y\\
  \{1,4\} &y & y &x  & z  & y &  y\\
  \{2,3\} &y & y & z & x   & y &  y\\
  \{2,4\} &y & z & y &  y &x  &  y\\
  \{3,4\} & z& y & y &  y &y  & x \\
\end{block}
\end{blockarray},}
\end{align*}
\noindent for some~$x,y,z \in \R$. The theorem states in this case that~$A^{(n)} \succeq 0$ for all~$n \geq 0$ if and only if~$x-2y+z\geq 0$,~$z \geq 0$ and~$y-z \geq 0$. 

Note that~$A^{(4)}$ has eigenvalues $x-z$ (multiplicity 3, eigenvectors~$v_{e_i} - v_{e_j}$ for the pairs of edges~$e_i,e_j$ with~$e_i \cap e_j = \emptyset$, with~$v_{e_i}$  the characteristic vector in~$\R^E$ with a~$1$ in position~$e_i$ and zeros elsewhere), $x+4y+z$ (multiplicity~$1$, eigenvector the all-ones vector) and~$x-2y+z$ (multiplicity 2, eigenvectors $(0, 1, -1,-1,1,0)\T$ and $(1,0,-1,-1,0,1)\T$). 

These eigenvalues are all at least zero if~$x-2y+z \geq 0$, $z\geq 0$, and $y-z \geq 0$, as is easy to check (and as follows from the theorem). The converse does not hold: for, e.g., $(x,y,z)=(3,1,2)$ all eigenvalues of~$A^{(4)}$ are positive, but~$y-z <0$. So this shows that~$A^{(4)} \succeq 0$ is necessary but not sufficient to conclude that~$A^{(n)} \succeq 0$ for all~$n \in \N$. 
\end{remark} 

We now briefly sketch how we use Theorem~\ref{Antheorem}  to prove Theorem~\ref{theorem2}. Details will be given in Section~\ref{howtouse}. 
We start in the same way as Brosch, Laurent and Steenkamp \cite{fdpoly}, by expressing the Hessian of~$f_d$ as a matrix polynomial of degree~$d-2$ in the standard monomial basis. Each coefficient matrix~$Q_{\gamma(n)} \in \R^{E \times E}$ can be seen to  depend on~$n$ and a multigraph with~$d-2$ fixed edges which is independent of~$n$. In order to prove convexity of~$f_d$ over~$\Delta_m$, it suffices to show that each of these matrices~$Q_{\gamma(n)}$ is positive semidefinite. 

To do this, we enumerate for a fixed~$d$ all multigraphs on~$d-2$ edges up to isomorphism. Such graph has $k \leq 2(d-2)$ vertices. For each multigraph, we construct a sequence~$(A^{(n)})_{n \geq k}$ (with~$A^{(n)} =Q_{\gamma(n)}$) which satisfies the conditions of Theorem~\ref{Antheorem}. With this theorem, we verify positive semidefiniteness of~$Q_{\gamma(n)}$, for all~$n \geq k$. After performing this verification for each multigraph with~$d-2$ edges up to isomorphism, we may conclude that the Hessian of~$f_d$ is positive semidefinite over~$\Delta_m$ (i.e.,~$f_d$ is convex over~$\Delta_m$)  \emph{for all~$n$}, for this fixed~$d$. We follow the described procedure for each fixed~$d \leq 9$ to exhibit the result in Theorem~\ref{theorem2}.

\subsection{Motivation}

The polynomial~$f_d$ arises naturally in a model of job scheduling with redundancy. The question whether the polynomial~$f_d$ is minimized at the uniform probability vector (in particular for~$L=2$) was posed by Cardinaels, Borst and Van Leeuwaarden~\cite{fdpolyoriginal} and answered affirmatively for~$d=2$ (for all~$L$) and~$d=3$ (for~$L=2$) by Brosch, Laurent and Steenkamp~\cite{fdpoly}.  Here we  give a sketch of the application in queueing theory, to motivate why we are studying this polynomial. For a detailed explanation about the model and about the continuous-time Markov chain setting (constituting the queueing theory framework) we refer to~\cite{stationary}, and for other details about the question about~$f_d$ we refer to~\cite{fdpolyoriginal}.

Suppose there are~$n$ parallel servers, which process jobs with speed~$\mu$. Jobs arrive as a Poisson process of rate $n \lambda$, for some~$\lambda >0$. When a job arrives,~$L$ replicas of the same job are sent with probability~$x_e$ to a subset~$e \subseteq [n]$ of~$L$ servers. When one of the $L$ replicas of a job finishes on a server, the replicas of it on the other servers are instantly abandoned. The servers process the jobs on a first come, first serve basis.

In the literature it is often assumed that the set of servers to which replicas of a job are sent is chosen uniformly at random. By contrast, Cardinaels, Borst and Van Leeuwaarden~\cite{fdpolyoriginal} investigate the impact of selecting the set of servers according to a specified probability distribution~$(x_e)_{e \in E}$. 
 They show that in the heavy-traffic regime (i.e., $\lambda \uparrow \mu$), this impact is relatively limited. However, in the light-traffic regime (i.e., $\lambda \downarrow 0$), the system occupancy is considerably more sensitive to the choice of the probability distribution~$(x_e)_{e \in E}$. They especially consider the case~$L=2$ in the light-traffic regime. The special case that the edge probabilities are uniform corresponds to the \emph{commonly considered power-of-two policy}.

As observed in~\cite{stationary}, the state (total occupancy) of the system at time~$t$ can be denoted by a single queue, i.e., as a vector $(e_1,\ldots,e_M) \in E^M$ where~$M=M(t)$ denotes the total number of jobs in the system, and~$e_i \in E$ is the set of servers to which copies of the~$i$-th oldest job have been assigned. So the newest job corresponds to~$e_M$. An example is given in Figure~\ref{examplefigure}.

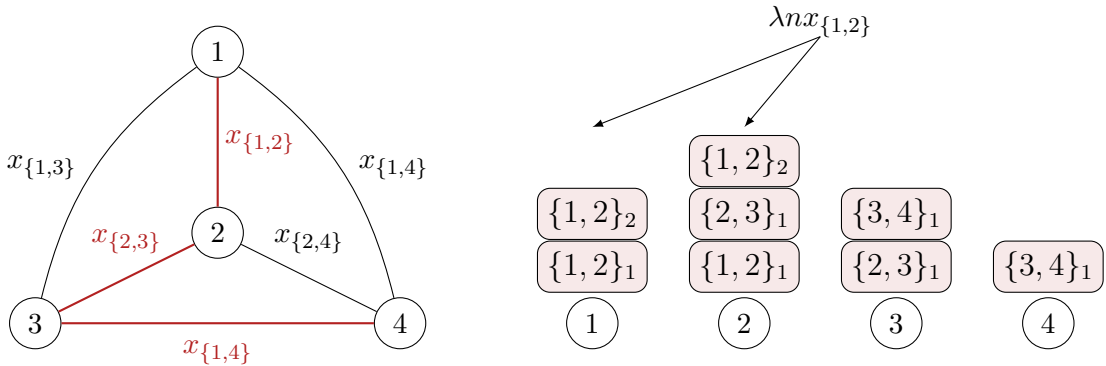
\begin{figure}[ht]
\centering
\begin{subfigure}[b]{0.49\textwidth}
\centering
\begin{tikzpicture}  [scale=1.2,auto=center,every node/.style={circle,draw=black},baseline=0pt] 
  \node (1) at (2,3)  {1}; 
  \node (2) at (2,1)  {2}; 
  \node (3) at (0,0)  {3}; 
  \node (4) at (4,0)  {4}; 
\path[-] 
(1)  edge [firebrick,thick]  node[right=-3,scale=1,draw=none]  {$x_{\{1,2\}}$}         (2)
(1)  edge [black,bend right=20]  node[left=1.5,scale=1,draw=none]  {$x_{\{1,3\}}$}         (3)
(1)  edge [black,bend left=20]  node[right=1.5,scale=1,draw=none]  {$x_{\{1,4\}}$}         (4)
(3)  edge [firebrick,thick]  node[below=-8,scale=1,draw=none]  {$x_{\{1,4\}}$}         (4)
(2)  edge [firebrick,thick]  node[above=-5,scale=1,draw=none]  {$x_{\{2,3\}}$}         (3)
(2)  edge [black]  node[above=-5,scale=1,draw=none]  {$x_{\{2,4\}}$}         (4);
\end{tikzpicture}
\end{subfigure}\vspace{-15pt}
\begin{subfigure}[t]{0.49\textwidth}
  \centering       
\begin{tikzpicture}[>=latex,scale =1,baseline=0pt] 
\tikzstyle{vertex} = [circle,draw=black, fill = none,scale = 1]
\tikzstyle{job} = [scale=1,rectangle,draw=black,fill = firebrick!10, rounded corners=1ex]
\node[vertex] (v1) at (-3,0) {1};
\node[vertex] (v2) at (-1,0) {2};
\node[vertex] (v3) at (1,0) {3};
\node[vertex] (v4) at (3,0) {4};
\begin{scope}[shift={(-3,0.75)}]
\node[job] at (0,0) {\large{$\{1,2\}_1$}};
\node[job] at (0,0.7) {\large $\{1,2\}_2$};
\end{scope}
\begin{scope}[shift={(-1,.75)}]
\node[job] at (0,0.0) {\large{$\{1,2\}_1$}};
\node[job] at (0,0.7) {\large $\{2,3\}_1$};
\node[job] at (0,1.4) {\large $\{1,2\}_2$};
\end{scope}
\begin{scope}[shift={(1,0.75)}]
\node[job] at (0,0.0) {\large{$\{2,3\}_1$}};
\node[job] at (0,0.7) {\large $\{3,4\}_1$};
\end{scope}
\begin{scope}[shift={(3,0.75)}]
\node[job] at (0,0.0) {\large{$\{3,4\}_1$}};
\end{scope}
\draw[->] (0,3.8) --(-3,2.6);
\draw[->] (0,3.8) -- (-1,2.6);
\node at (0,4) {$\lambda n x_{\{1,2\}}$};
\end{tikzpicture}
\end{subfigure}
\caption{\small On the left, the complete graph on~$n=4$ vertices is shown. The $4$ vertices represent the servers and the edges represent the possible pairs of servers to which  job replicas are sent. On the right, the server occupancy for the state~$(\{1,2\}, \{2,3\},\{3,4\},\{1,2\})$ is presented. The~$k$-th arrival of a job which is sent to servers~$\{i,j\}$ is denoted by~$\{i,j\}_k$ (cf.\ Figure~$1$ of~\cite{fdpolyoriginal}), so we can write the above state also as $(\{1,2\}_1, \{2,3\}_1,\{3,4\}_1,\{1,2\}_2)$.}\label{examplefigure}
\end{figure}

It was proved in~\cite{stationary} that if the service times of the jobs are independently and exponentially distributed with unit mean, under a certain necessary and sufficient condition for stability of the system (Theorem 1 of~\cite{stationary}, see also equation~(2) of~\cite{fdpolyoriginal}), and assuming all servers have the same speed~$\mu$,  the limiting (stationary) probability of being in state $(e_1,\ldots,e_M) \in E^M$ equals
\begin{align}
\pi(e_1,\ldots,e_M)= C \prod_{i=1}^M \frac{n \lambda x_{e_i}}{\mu |e_1 \cup \ldots \cup e_i| },
\end{align} 
for a normalization constant~$C >0$. %, chosen such that~$\pi$ is a probability distribution. 
The authors of~\cite{stationary} emphasize that~$\pi(e_1,\ldots,e_M)$ cannot be written as a product of independent per-server or per-edge terms, as the denominator for the~$i$-th job depends on all jobs already in the system, i.e., on~$e_1,\ldots,e_i$. 

Now let~$Q_{\lambda}(x)$ denote a random variable with the stationary distribution of the number of jobs in the system with edge selection probabilities~$x=(x_e)_{e \in E}$. Then %the probability that there are~$d$ jobs in the system is 
\begin{align}
\mathbb{P}(Q_{\lambda}(x)=d) = \sum_{(e_1,\ldots,e_d) \in E^d} \pi (e_1,\ldots,e_d),
\end{align} 
which is equal to~$f_d(x)$ up to a scalar multiple. The authors of~\cite{fdpolyoriginal} specifically consider the two-server light traffic regime, i.e., the case that $L=2$ and~$\lambda \downarrow 0$. They show that in this case $\mathbb{P}(Q_{\lambda}(x)=0) = C = 1 +o(1)$, and 
\begin{align}
\mathbb{P}(Q_{\lambda}(x) \geq d) = \left( \frac{n\lambda}{\mu }\right)^d f_d(x) + o(\lambda^d), 
\end{align} 
Hence, letting ~$x^*=(1,\ldots,1)/|E|$ denote the uniform probability vector, one has
\begin{align} \label{motivationlimit}
\lim_{\lambda \downarrow 0} \frac{\mathbb{P}(Q_{\lambda}(x^*) \geq d ) }{\mathbb{P}(Q_{\lambda}(x) \geq d )} =\lim_{\lambda \downarrow 0} \frac{ f_d(x^*) + o(1) }{ f_d(x) + o(1) }.
\end{align}
 So if the polynomial~$f_d$ attains its minimum at the uniform probability vector~$x^*$, then the limit in~\eqref{motivationlimit} is at most~$1$. This means that in this case in the light-traffic regime, to minimize the system occupancy, one should select the assignments  of the job replicas to the servers uniformly at random. This motivates the question to show that the polynomial~$f_d$ attains its minimum over the standard simplex at the uniform distribution~$x^*$. %Cardinaels, Borst and Van Leeuwaarden assert that it seems to be a \emph{huge combinatorial challenge}  to compute the probability that there are exactly~$d$ jobs in the system, and that establishing that the number of jobs in the system is minimized for uniform edge probabilities \emph{thus remains a challenge for further research}.

It was proved by Brosch, Laurent and Steenkamp that for~$d \leq 3$ the statement holds~\cite{fdpoly}, and the results of the current paper imply that it is true for~$d \leq 9$.

\subsection{Background: exploiting the symmetry of the problem}
To show convexity of~$f_d$ for~$d \leq 9$ we will exploit its symmetry properties.  Brosch, Laurent and Steenkamp also used the symmetry of~$f_d$ in~\cite{fdpoly}. They started with the following observation, based on symmetry.
\begin{lemma}[\cite{fdpoly}]
For any~$L,d,n\geq 2$, if the polynomial~$f_d$ from~\eqref{fdeq} is convex over the standard simplex~$\Delta_m$, then it attains its global minimum over~$\Delta_m$ at  the uniform probability vector~$x^* = \tfrac{1}{m}(1,\ldots,1)$. 
\end{lemma}
 To see this, note that if $f$ is convex and $f(g\cdot x) = f(x)$ for all $g\in G$, then $f$ has a minimizer that satisfies $g \cdot x = x$ for all $g\in G$. In the present setting, $f_d$ is invariant under the action of $S_n$ and~$x^*$ is the only fixed point in~$\Delta_m$ of the action of $S_n$. 

Brosch, Laurent and Steenkamp~\cite{fdpoly} first considered a related symmetric polynomial, defined as follows (again for integers~$n,d,L \geq 2$):
\begin{align}\label{pd}
p_d(x) := \sum_{(e_1,\ldots,e_d) \in E^d} \frac{1}{|e_1 \cup \ldots \cup e_d|}x_{e_1}\cdots x_{e_d}.
\end{align}
They proved that~$p_d$ is convex for all~$n,d \geq 2$ and for all edge cardinalities~$L$, and that convexity implies that~$p_d$ attains its minimum at the uniform probability vector. They observed that~$p_2 = L f_2$, and hence also proved that~$f_2$ is convex for all edge cardinalities~$L$. The polynomials~$p_d$ are interesting in themselves, and helped to give insight in the question of Cardinaels, Borst and Van Leeuwaarden about~$f_d$.

Symmetry is a frequently used tool in semidefinite optimization. We give some background and references. The starting point is Wedderburn-Artin theory (cf.~\cite{wedderburn}), from which it follows that matrix~$*$-algebras can be block-diagonalized. The technique we will use to obtain the first part of our symmetry reduction~(Proposition~\ref{poslemma} below) is a special case of this.  
The linear programming bound by Delsarte~\cite{delsarte} is an early example of the usage of symmetry, which bound was discovered to be a symmetry reduced (and slightly strengthened) version of Lov\'asz's theta-function (cf.~\cite{schrijverdelsarte}). 

Group symmetry can be used to reduce the size of matrices in semidefinite programs. This was studied in landmark papers by Kanno, Ohsaki, Murota and Katoh~\cite{kanno}  and Gatermann and Parrilo~\cite{gatpar}.   For an exposition about the usage of symmetry in semidefinite programming, we refer to~\cite{invsemi}. Examples of concrete areas of application are crossing numbers of complete bipartite graphs (cf.~\cite{regular}), approximations for quadratic assignment problems (cf.~\cite{quad}) and truss topology optimization (cf.~\cite{truss}).

Another specific area of application is coding theory, in which several block-diagonalizati-ons have been developed to compute semidefinite bounds based on tuples of codewords. This approach was pioneered by Schrijver~\cite{schrijver}. At a fixed level in the hierarchy, these bounds can be computed in polynomial time, cf.\ Laurent~\cite{laurent}. In earlier works, direct and analytical derivations of the reductions of the involved algebras (the Terwiliger algebras of the Hamming and Johnson schemes) were given~\cite{tanaka,quadruples, laurent, schrijver, vallentin}. The Terwilliger algebra of the Hamming scheme is the algebra of $S_n$-invariant matrices in~$\C^{2^n \times 2^n}$ indexed by the collection of all subsets of~$[n]$. The algebra of~$S_n$-invariant matrices indexed by~$\tbinom{[n]}{L}$ for a fixed~$L$ is called the Terwilliger algebra of  the Johnson scheme. We consider~$L=2$, but our matrices are only symmetric under~$S_{n-k}$ (for~$n\geq k$), where~$k$ is fixed. Brosch, Laurent and Steenkamp~\cite{fdpoly} used the known block-diagonalizations of the Terwilliger algebras as a tool  in their proof of convexity of~$p_d$ (for all~$d$) and of~$f_d$ for~$d\leq 3$. 

More recently, in coding theory an approach based on elementary representation theory ---with as main element an explicit decomposition into irreducible submodules of a certain associated vector space--- has been used to obtain symmetry reductions of semidefinite programs for finding upper bounds on the cardinalities of several types of codes~\cite{gijswijt, artikel, mixed, cw4, leeartikel} (see~\cite{proefschrift} as background for a unified treatment).
In the present work, we will use this approach based on representation theory of finite groups to obtain the first part of the reduction. 

Symmetry and representation theory have extensively been used as a tool to answer algebraic questions about polynomials. With regard to semidefinite programming, the central question is then whether certain symmetric polynomials can be written as sums of squares. We point the reader also to references~\cite{gatpar, thomas, thomas2,riener} about finding these sum of squares certificates for symmetric polynomials. In particular, symmetric sums of squares (in~$n$ variables under the action of~$S_n$) have been characterized~\cite{blekherman}. The recent works~\cite{thomas, thomas2} consider symmetric polynomials with variables indexed by the $L$-subsets (of~$n$) hypercube. In the mentioned references a main focus is to show that, using representation theory of the symmetric group (cf.~\cite{sagan}), the sizes of the involved semidefinite programs can be made independent of~$n$. However, the number $n$ still might appear in the coefficients of the reduced matrices. The authors also give some examples of asymptotic results and find links to Razborov's theory of flag algebras~\cite{razborov}.

In the present paper the main tool (given in Theorem~\ref{Antheorem}) is a symmetry reduction to show that certain symmetric matrices are positive semidefinite for all~$n$ if and only if three matrices with order \emph{and coefficients} independent of~$n$ are positive semidefinite.

\subsection{Outline of the paper}
In Section~\ref{howtouse}, we explain how  Theorem~\ref{Antheorem} will be used to show Theorem~\ref{theorem2}.  In Section~\ref{preliminaries} we give the necessary preliminaries to prove Theorem~\ref{Antheorem}. In Section~\ref{bigproof} we proceed to give the proof of Theorem~\ref{Antheorem}, which is the most important part of this paper. Afterwards we prove in Section~\ref{smallpropproof} the following small proposition (here we fix a sequence $\gamma \in \N_{d-2}^m$, as in the previous section):
\begin{proposition} \label{drieprop}
Let~$(A^{(n')})_{n' \geq k}$ be the sequence of matrices corresponding to~$\gamma$. Then
$$
a_{(\{k+1,k+2\},\{k+1,k+2\})} - 2  a_{(\{k+1,k+2\},\{k+1,k+3\})} +  a_{(\{k+1,k+2\},\{k+3,k+4\})} \geq 0,
$$
where we write~$a_{(e_i,e_j)}$ for the~$e_i,e_j$-th entry of~$A^{(k+4)}$, for~$e_i, e_j \in \tbinom{[k+4]}{2}$.
\end{proposition} 
This proposition implies, together with Theorem~\ref{Antheorem}, that~$A^{(n')}$ is positive semidefinite for all~$n' \geq k$ if and only if the two block matrices from Theorem~$\ref{Antheorem}$ are positive semidefinite (where ~$(A^{(n')})_{n' \geq k}$ denotes the sequence of matrices corresponding to~$\gamma$).

In Section~\ref{computer} we give computational results of the computer search which exhibits the result in Theorem~\ref{theorem2}, using Theorem~\ref{Antheorem}. We conclude the paper with a short discussion (in Section~\ref{extensions}) of the possible directions in which the main ideas of this paper could be extended.

\section{How to use Theorem~\ref{Antheorem} to show Theorem~\ref{theorem2}}\label{howtouse}
In this section we explain how Theorem~\ref{Antheorem} will be used to exhibit the result in Theorem~\ref{theorem2}. The polynomial~$f_d$ is convex over~$\Delta_m$ if its Hessian
\begin{align}
H(f_d) = \left( \frac{\partial^2 f_d}{\partial e_i \partial e_j} \right)_{e_i,e_j \in E}
\end{align}
 is positive semidefinite on~$\Delta_m$. We begin in the same way as Brosch, Laurent and Steenkamp \cite{fdpoly}, by writing the Hessian  of~$f_d$ as a  matrix polynomial, which involves a set of matrices~$Q_{\gamma}$ arising as the coefficients of the Hessian in the monomial base.  The Hessian of~$f_d$ is positive semidefinite on~$\Delta_m$  if each of the matrices~$Q_{\gamma}$ appearing in this matrix polynomial is positive semidefinite. For a fixed~$d$, the matrices~$Q_{\gamma}$ which arise can be seen to correspond to multigraphs on~$d-2$ edges. Then we use our new result, Theorem~$1.3$, to verify positive semidefiniteness of each of these matrices.
 We verify the conjecture for~$d \leq 9$ and for all~$n$ (in case~$L=2$). 

Let us first recall some notation and results from~\cite{fdpoly}. For a~$d$-tuple~$(e_1,\ldots,e_d) \in E^d$, define
$$
c_{(e_1,\ldots,e_d)} := \frac{1}{|e_1 \cup \ldots \cup e_d|}. 
$$
Let~$\N^m_d$ denote the set of~$m$-tuples in~$\Z_{\geq 0}$ with sum of entries equal to~$d$ and we set~$|\alpha|:=\sum_{i=1}^n \alpha_i$ for~$\alpha \in \N^m$. For a~$d$-tuple~$\underline{e}=(e_{i_1},\ldots ,e_{i_d}) \in E^d$, with~$i_1, \ldots,i_d \in [m]$, define the sequence~$\alpha_n(\underline{e}) \in \N^m$, where $\alpha_n(\underline{e})_l$ is the number of indices among~$i_1,\ldots ,i_d$ that are equal to~$l$, for each~$l \in [m]$. Then
$$
x_{e_{i_1}} \cdots x_{e_{i_{d}}}  = x_{e_1}^{\alpha_n(\underline{e})_1} \cdots x_{e_m}^{\alpha_n(\underline{e})_m}, 
$$
and $|\alpha_n(\underline{e})|=d$ so that $\alpha_n(\underline{e}) \in \N_d^m$. 
For $\alpha \in \N_d^m$ consider any~$d$-tuple~$\underline{e}\in E^d$ such that~$\alpha_n(\underline{e})=\alpha$ and define
$
\widehat{c_{\alpha}} := c_{\underline{e}}$.
Similarly, for $\alpha \in \N^m_d$, define
\begin{align}\label{balphadef}
b_{\alpha} :=  \sum_{\substack{(e_1,\ldots,e_d) \in E^d\\ \alpha_n(\underline{e})=\alpha}} \prod_{i=1}^d \frac{1}{|e_1 \cup \ldots \cup e_i|}, \quad \text{ so that } f_d(x) = \sum_{\alpha \in \N^m_d} b_{\alpha} x^{\alpha}.
\end{align}
For~$i \in \tbinom{[n]}{2}$, write~$v_i$ for the indicator vector in~$\R^{\tbinom{[n]}{2}}$, which has a~$1$ in position~$i$ and zeros elsewhere. Then the Hessian~$H(f_d)$ is equal to
$$
H(f_d)(x) = \sum_{\gamma \in \N_{d-2}^m} x^{\gamma} Q_{\gamma},
$$
where
\begin{align} \label{qgammadef}
(Q_{\gamma})_{i,j} = \begin{cases} (\gamma_i+1) (\gamma_j+1) b_{\gamma+v_i+v_j} &\mbox{if } i \neq j, \\
 (\gamma_i+1) (\gamma_i+2) b_{\gamma+2 v_i } & \mbox{if } i=j, \end{cases}
\end{align}
for~$i,j \in \tbinom{[n]}{2}$.
For any~$\alpha \in \N_d^m$, we write $\alpha! = \alpha_1 ! \cdots \alpha_m !$,
and define
$$
\widehat{b}_{\alpha} := \alpha! \,\, b_{\alpha}.
$$
Then one has (cf.~\cite[Lemma 11]{fdpoly}), for~$d \geq 3$ and~$\gamma \in \N_{d-2}^m$, that
\begin{align}\label{qgammadef2}
Q_{\gamma} = \tfrac{1}{\gamma!} \left(\widehat{b}_{\gamma + v_i+ v_j}\right)_{i,j=1}^m.
\end{align}
The following recurrence relations are given in~\cite{fdpoly}:
\begin{align}
\widehat{b}_{\alpha} = \widehat{c}_{\alpha} \sum_{k: \alpha_k \geq 1} \alpha_k \widehat{b}_{\alpha- v_k},
\end{align}
and (cf.~\cite[Lemma 12]{fdpoly})
\begin{align}\label{moniquedecomp}
Q_{\gamma} =\underbrace{(\widehat{c}_{\gamma+v_i+v_j})_{i,j=1}^m}_{=: M_{\gamma} } \circ \left( \sum_{k \in [m]: \gamma_k \geq 1}  Q_{\gamma -v_k} + \underbrace{\frac{1}{\gamma!} \left(\widehat{b}_{\gamma + v_i}+ \widehat{b}_{\gamma + v_j}\right)_{i,j=1}^m }_{=: R_{\gamma}} \right).
\end{align}

Brosch, Laurent and Steenkamp~\cite{fdpoly} showed that the matrix~$M_{\gamma}=(\widehat{c}_{\gamma+v_i+v_j})_{i,j=1}^m$ in~\eqref{moniquedecomp} is positive semidefinite for all~$n,d,L \geq 2$.  The matrices~$M_{\gamma}$ appear as coefficients of the Hessian of the polynomial~$p_d$ from~\eqref{pd} expressed in the monomial base. By showing positive semidefiniteness of the matrices~$M_{\gamma}$  they proved convexity of~$p_d$ over~$\Delta_m$ for all~$n,d,L \geq 2$, and hence also of~$f_2=\tfrac{1}{L}p_2$ for all~$d,L \geq 2$. Moreover, they provided numerical experiments, showing that the term between brackets in~\eqref{moniquedecomp} is positive semidefinite f or several specific small cases of~$d,L,n$ and that~$R_{\gamma}$ in many of these cases has a strictly negative smallest eigenvalue.

We now proceed to describe a sequence of matrices (related to the matrices~$Q_{\gamma}$) to which we can apply Theorem~\ref{Antheorem}. 
Fix a sequence $\gamma \in \N_{d-2}^m$. We aim to prove that~$Q_{\gamma} \succeq 0$. Let~$\underline{e}=(e_1,\ldots,e_{d-2}) \in E^{d-2}$ be any $(d-2)$-tuple with~$\alpha_n(\underline{e})  = \gamma$. Set~$k:=|\cup_{i=1}^{d-2} e_i|$.  Without loss of generality, we may assume that the vertices occuring in the edges in the tuple~$\underline{e}$ are contained in~$[k]\subseteq [2(d-2)]$, otherwise we relabel the vertices in~$[n]$. Note that~$n \geq k$.

For each~$n' \geq k$, define~$m' := \tbinom{n'}{2}$ and set $\gamma(n'):=\alpha_{n'}(\underline{e}) \in   \N_{d-2}^{m'}$.  
Furthermore, define~$A^{(n')}:= Q_{\gamma(n')}$, which is a matrix of order~$m' \times m'$. (Here~$Q_{\gamma(n')}$ is defined as in~\eqref{qgammadef}, except that we take~$n=n'$, $m=m'$ and~$E=\tbinom{[n']}{2}$ in the preceding definitions.) Then by construction we have $A^{(n)} = Q_{\gamma}$. Furthermore, we can view each~$A^{(n')}$ as an~$\tbinom{[n']}{2} \times \tbinom{[n']}{2}$-matrix. We will call the sequence~$(A^{(n')})_{n'\geq k}$ \emph{the sequence of matrices corresponding to} $\gamma$.

\begin{examp}\label{examplegamma}
Let~$d=3$, let~$n \in \N$ and let~$\gamma = \alpha_n((\{1,2\})) \in \N_{d-2}^m$. As above, we set $k:=|\cup_{i=1}^{d-2} e_i|$. Let~$n' \geq 2$, $m':= \tbinom{n'}{2}$ and $i, j \in \tbinom{[n']}{2}$. Then~$\gamma(n') = \alpha_{n'}((\{1,2\})) \in \N^{m'}_{d-2} = \N^{m'}_{1}$ is the sequence with a~$1$ in the position corresponding to the edge~$e_1:=\{1,2\}$ and zeros elsewhere. Then, for all~$i,j \in \tbinom{[n']}{2}$, we have
\begin{align} \label{easycase}
A^{(n')}_{i,j} &= \hat{b}_{\gamma(n') +v_i+v_j}= (\gamma(n') +v_i+v_j)! \sum_{\substack{(e,f,g) \in E^3 :\\\{e,f,g\}=\{e_1,i,j\} }}\frac{1}{|e|}\cdot\frac{1}{|e\cup f|} \cdot\frac{1}{|e \cup f \cup g|} \notag \\
&=\frac{(v_{e_1} +v_i+v_j)!}{2|e_1 \cup i \cup j|}\sum_{\substack{(e,f,g) \in E^3 :\\\{e,f,g\}=\{e_1,i,j\} }}\frac{1}{|e\cup f|} 
= \frac{1}{|e_1 \cup i \cup j|}\left( \frac{1}{|e_1\cup i|} + \frac{1}{|e_1\cup j|} + \frac{1}{|i\cup j |}\right),
\end{align}
where the last equality is verified by distinguishing the cases that~$e_1,i,j$ are all pairwise different, two of them are equal and one is different, or they are all equal. So, for the sequence~$(A^{(n')})_{n'\geq k}$ the first two properties of Theorem~\ref{Antheorem} hold. The third property also holds: each~$A^{(n')}$ is invariant under the simultaneous action of~$S_{n'-2}$ on its rows and columns (induced by the action of~$S_{n'-2}$ on~$\{3,\ldots,n'\}$). This is seen from~\eqref{easycase}: $A^{(n')}_{i,j} = A^{(n')}_{\sigma \cdot i,\sigma \cdot j}$ for all~$n'\geq 2$ and~$i,j \in \tbinom{[n']}{2}$, as for each subset~$U \subseteq [n']$ and~$\sigma \in S_{n'-2}$  we have $|U|= |\sigma \cdot U|$ and~$|e_1 \cup U| = |e_1 \cup \sigma \cdot U|$.

To verify the conjecture for~$d=3$ and for all~$n$, we can now apply Theorem~\ref{Antheorem} to the sequence~$(A^{(n)})_{n \geq 2}$, and check that the three block matrices from this theorem are positive semidefinite, cf.\ Appendix~\ref{d3verify}.
\end{examp}

The argument for~$S_{n'-k}$-invariance of~$A^{(n')}$ of the above example holds for general~$d$ and~$\gamma$. 
\begin{proposition}
Fix a sequence $\gamma \in \N_{d-2}^m$. Let~$(A^{(n')})_{n'\geq k}$ be the sequence of matrices corresponding to~$\gamma$. Then the three conditions of Theorem~\ref{Antheorem} hold for~$(A^{(n')})_{n'\geq k}$.
\end{proposition}
\proof

Let~$\underline{e}=(e_1,\ldots,e_{d-2}) \in E^{d-2}$ be any $(d-2)$-tuple with~$\alpha_n(\underline{e})  = \gamma$. Set~$k:=|\cup_{i=1}^{d-2} e_i|$. 
 Without loss of generality, we may assume that the vertices occuring in the edges in the tuple~$\underline{e}$ are contained in~$[k]\subseteq [2(d-2)]$, otherwise we relabel the vertices in~$[n]$. Note that~$n \geq k$. 
The first two statements of Theorem~\ref{Antheorem} hold trivially for~$(A^{(n')})_{n'\geq k}$, from the definition in~\eqref{qgammadef}. 

We prove that each~$A^{(n')}$ is invariant under simultaneous permutation of the rows and columns by~$S_{n'-k}$, i.e.,
$$
A^{(n')}_{i,j} = A^{(n')}_{\sigma \cdot i, \sigma \cdot j}  \text{ for all } \sigma \in S_{n'-k}, \text{ and } i, j \in \tbinom{[n']}{2}.
$$
For any~$\sigma \in S_{n'-k}$ the sequence~$\gamma(n'):=\alpha_{n'}(\underline{e}) \in   \N_{d-2}^{m'}$ is fixed by~$\sigma$, since the edges in the tuple~$\underline{e}$ are contained in~$[k]$. Hence it suffices (cf.~\eqref{qgammadef2})  to prove that
\begin{align} \label{bequality}
 b_{\gamma(n')+v_{\sigma(i)}+v_{\sigma(j)}} = b_{\gamma(n')+ v_i+v_{j}},
\end{align}
for all  $\sigma \in S_{n'-k}, i, j \in \tbinom{[n']}{2}$ and~$\sigma \in S_{n'-k}$.
Write~$i=\{i_1,i_2\}$, $j=\{j_1,j_2\}$. Then each summand in the definition~\eqref{balphadef} of~$ b_{\gamma(n')+v_{\sigma(i)}+v_{\sigma(j)}}$  is a product of terms where each term has the form
\begin{align*}
\frac{1}{|U|},  \,\,\,\, \frac{1}{| U \cup \sigma \cdot \{i_1,i_2\}|} \,\,\, \text{ or } \frac{1}{| U \cup\sigma \cdot \{ i_1,i_2\}  \cup \sigma \cdot \{ j_1,j_2\}|},
\end{align*}
for~$U \subseteq [k]$. But for any~$U \subseteq [k]$ and~$W \subseteq [n]$ we have~$|U \cup \sigma W| = |\sigma (U \cup W)| = |U \cup W|$, as~$U$ is fixed by~$\sigma$. So
\begin{align*}
| U \cup \sigma \cdot \{i_1,i_2\}| &=| U \cup \{i_1,i_2\}|,\quad\text{ and}\\
| U \cup\sigma \cdot \{ i_1,i_2\}  \cup \sigma \cdot \{ j_1,j_2\}| &= | U \cup\{ i_1,i_2\}  \cup  \{ j_1,j_2\}|, \end{align*}
for~$U \subseteq [k]$.  This shows~$\eqref{bequality}$. 
\endproof

\begin{procedure}\label{procver}
One can now perform the following steps to verify positive semidefiniteness of~$f_d$ for a fixed~$d$, for all~$n$ (and hence to exhibit the result in Theorem~\ref{theorem2}).
\begin{enumerate}[label=(\roman*)]
\item Enumerate up to isomorphism all multigraphs $\underline{e}= (e_1,\ldots,e_{d-2})$ with exactly~$(d-2)$ edges without isolated vertices. (Such graph has at most~$2(d-2)$ vertices.)
\item For each multigraph~$\underline{e}= (e_1,\ldots,e_{d-2})$ from the previous step, we set~$k:= |\cup_{i=1}^{d-2} e_i|  \leq 2(d-2)$. We relabel the vertices such that~$e_i \subseteq [k]$ for each~$i \in [d-2]$.
\item  For each~$n \geq k$, set~$\gamma(n) := \alpha_n(\underline{e})\in \N_{d-2}^{\binom{n}{2}}$ and define~$A^{(n)} := Q_{\gamma(n)}$. We verify that~$A^{(n)}$ is positive semidefinite for each~$n \geq k$ by verifying that the three  block matrices given in Theorem~\ref{Antheorem} (which are constructed from the matrix~$A^{(k+4)}$) are positive semidefinite.
\end{enumerate}
After performing these steps, we know that $Q_{\gamma(n)}$ is positive semidefinite for each~$n \geq k$ and for each multigraph, hence~$H(f_d) \succeq 0$, hence~$f_d$ is convex (for the considered fixed~$d$, but for all~$n$). 
\end{procedure}

\section{General theory: representative sets\label{preliminaries}}
We now indicate a known general method for symmetry reduction using elementary representation theory (which method was used in~\cite{artikel}). We briefly state the results.

Let~$G$ be a group acting on a finite dimensional complex vector space~$V$. Then~$V$ is called a~$G$\emph{-module}. A \emph{submodule} $U$ of~$V$ is a subspace of~$V$ which is itself~$G$-invariant. For two~$G$-modules~$V$ and~$W$, a $G$\emph{-homomorphism} $\psi: V \to W$ is a linear map such that~$g \cdot \psi(v)=\psi(g \cdot v)$ for all~$v \in V$ and $g \in G$. If there exists a bijective $G$-homomorphism from~$V$ to~$W$ (also called a $G$\emph{-isomorphism}) then~$V$ and $W$ are called \emph{equivalent} or ($G$-)\emph{isomorphic}.

A $G$-module~$V$ is called \emph{irreducible} if its only submodules are~$\{0\}$ and~$V$ itself and~$V \neq \{0\}$. The algebra of~$G$-homomorphisms~$V \to V$ is denoted by $\text{End}_G(V)$ and is called the  \emph{centralizer algebra} of the action of~$G$ on~$V$.

Let~$V$ be a finite dimensional complex vector space and~$G$ be a finite group acting on~$V$. In this case the~$G$-module~$V$ can be  decomposed as $V=V_1 \oplus \ldots \oplus V_k$ such that each~$V_i$ is can be written as~$V_{i,1} \oplus \ldots \oplus V_{i,m_i}$, with the property that the~$V_{i,j}$ are irreducible~$G$-modules such that~$V_{i,j}$ and~$V_{i',j'}$ are isomorphic if and only if~$i=i'$.  The number~$k$ is unique. Moreover, the~$V_i$ are unique up to permuting indices and they are called the~$G$-\emph{isotypical components}.  

Choose for each~$i \leq k$ and~$j \leq m_i$,    a nonzero vector~$u_{i,j} \in V_{i,j}$ with the property that for each~$i \in [k]$ and all~$j,j'\leq m_i$ there exists a~$G$-isomorphism~$V_{i,j} \to V_{i,j'}$ mapping~$u_{i,j}$ to~$u_{i,j'}$. For each~$i \leq k$,  define~$U_i:=(u_{i,1},\ldots ,u_{i,m_i})$, which is an ordered~$m_i$-tuple of elements~$u_{i,j}$ (where $j=1,\ldots,m_i$).
\begin{defn}[Representative set]
We call any set~$\{U_1,\ldots,U_k\}$ obtained as described in the previous paragraph a \emph{representative set} for the action of~$G$ on~$V$.
\end{defn}
The notion `representative set' is closely related to the notion of a~\emph{symmetry adapted basis}, a basis which respects the decomposition of~$V$ into~$G$-modules~\cite{symmetryadapted}. However, in a representative set one `representative' basis element of each irreducible module is chosen. The collection of basis elements~$u_{i,j}$ constituting the ordered~$m_i$-tuples in the representative set also is sometimes called a \emph{symmetry basis} in the literature (e.g., in~\cite{blekherman}).

Since, for each~$i,j$, it holds that~$V_{i,j}=$ is spanned by~$G \cdot u_{i,j}$, we have~$V_{i,j} = \C G \cdot u_{i,j}$, where~$\C G$ denotes the (complex) group algebra of~$G$. So
\begin{align} \label{Rm}
V = \bigoplus_{i=1}^k\bigoplus_{j=1}^{m_i} \C G \cdot u_{i,j}.
\end{align}
Moreover, it holds that (which follows from Schur's well-known lemma) that
\begin{align}
\dim\text{End}_G(V) = \dim\text{End}_G\left(\bigoplus_{i=1}^k\bigoplus_{j=1}^{m_i}V_{i,j}\right) = \sum_{i=1}^k m_i^2.
\end{align}

Let~$G$ be a finite group acting on a finite dimensional complex vector space~$V$. Let~$\langle \, , \rangle$ be a $G$-invariant inner product on~$V$. (From any inner product $\langle \, ,  \rangle$  on~$V$ one can derive a $G$-invariant inner product~$\langle \, , \rangle_G$ via~$\langle x ,y \rangle_G := \sum_{g \in G} \langle g \cdot x , g \cdot y \rangle$.)  Let~$\{U_1,\ldots,U_k\}$ be any representative set for the action of~$G$ on~$V$ and define the (linear) map
\begin{align} \label{PhiC}
    \Phi \colon \text{End}_G(V) \to \bigoplus_{i=1}^k \C^{m_i \times m_i} \,\, \text{ with } \,\, A \mapsto \bigoplus_{i=1}^k \left( \langle A u_{i,j'} , u_{i,j}  \rangle \right)_{j,j' =1}^{m_i}.
\end{align}
An element~$A \in \text{End}(V)$ is called \emph{positive semidefinite} if $\langle Av, v \rangle \geq 0$ for all~$v \in V$ and $\langle Av,w \rangle = \langle v, A w\rangle$ for all~$v,w \in V$. 
\begin{proposition}\label{poslemma}
The map~$\Phi$ is bijective, and for any $A \in \text{\normalfont End}_G(V)$:
\begin{align}
 \text{$A$ is positive semidefinite} \,\,\, \Longleftrightarrow  \,\,\, \text{$\Phi(A)$ is positive semidefinite}.
 \end{align}
\end{proposition}
This proposition follows from well-established general theory. For a proof of the proposition in this exact form, we refer to~\cite{proefschrift}.

In our case, we will apply the above to the following. For any set~$Z$ and group~$G$ acting on~$Z$ we write~$Z^G$ for the set of all~$G$-invariant elements of~$Z$. Fix~$k \in \N$. For~$n \in \N$, consider the complex vector space~$V:= \C^E$, where as before~$E = \tbinom{[n]}{2}$. Then the finite group~$G:=S_{n-k}$ acts on~$E$ as in the previous section, hence on~$V=\C^E$ via permutation matrices. This means that~$\text{End}_G(V)$ is naturally isomorphic to~$(\C^{E \times E})^{S_{n-k}}$, i.e., the space of complex~$E \times E$-matrices which are invariant under the simultaneous action of~$S_{n-k}$ on their rows and columns. In this case~\eqref{PhiC} specializes to
\begin{align} 
    \Phi \colon(\C^{E \times E})^{S_{n-k}} \to \bigoplus_{i=1}^k \C^{m_i \times m_i} \,\, \text{ with } \,\, A \mapsto \bigoplus_{i=1}^k U_i^* A U_i,
\end{align}
and since our representative sets will turn out to consist of real matrices, we can replace~$\C$ by~$\R$ in the above equation, and obtain a linear bijective map~$\Phi \colon(\R^{E \times E})^{S_{n-k}} \to \bigoplus_{i=1}^k \R^{m_i \times m_i}$ with the property that~$A \succeq 0$ if and only if~$\Phi(A)\succeq 0$ for all~$A \in (\R^{E \times E})^{S_{n-k}}$.

For convenience, we record the following characterization of representative sets  (cf. \cite[proof technique of Proposition 2]{artikel}, separately stated with proof in~\cite[Proposition 2.4.3]{proefschrift}). 
\begin{proposition}\label{propset}
Let~$G$ be a finite group acting on a finite dimensional $\C$-vector space~$V$. Let $k,$ $m_1,\ldots,m_k \in \N$, and $u_{i,j} \in V$ for~$i=1,\ldots,k$,~$j=1,\ldots,m_i$. Then 
\begin{align}\label{lemmaset1}
\left\{ \,\left(u_{i,1},\ldots ,u_{i,m_i}\right)\, \, | \,\, i=1,\ldots,k \right\}
\end{align}
is a representative  set for the action of~$G$ on~$V$ if and only if:
\begin{enumerate}[label=(\roman*)]
\item \label{1eig} $ V=  \bigoplus_{i=1}^k\bigoplus_{j=1}^{m_i} \C G \cdot u_{i,j}$,
\item \label{2eig} For each $i=1,\ldots, k$ and  $j,j' =1,\ldots, m_i$, there exists a $G$-isomorphism~$\C G\cdot  u_{i,j} \to \C G \cdot u_{i,j'}$ which maps~$u_{i,j}$ to~$u_{i,j' }$,
\item \label{3eig} $\sum_{i=1}^k m_i^2  \geq \dim (\text{\normalfont End}_G(V))$.
\end{enumerate}
\end{proposition}

\section{Proof of Theorem~\ref{Antheorem}: symmetry reduction\label{bigproof}}
%Throughout this section we assume that~$n \geq 2k$. 
Fix~$k \in \N$. For~$n \in \N$ with~$n \geq k$, let again $S_{n-k} \subseteq S_n$ be the subgroup of all~$\sigma \in S_n$ with~$\sigma(i)=i$ for all~$i \in [k]$. Set~$m=(m_1,m_2,m_3)= (\tbinom{k}{2}+k+1,k+1,1)$. For~$j \in [\tbinom{k}{2}]$, let~$e_j$ denote the $j$th subset of order 2 of~$[k]$, where these subsets are ordered in lexicographic (or: any fixed) order. 
Recall that for~$e \in E= \tbinom{[n]}{2}$, we write~$v_e$ for the corresponding standard basis vector in~$\C^E$.
\begin{proposition}\label{u1u2u3}
Suppose that~$n \geq k+4$. A representative set for the action of~$S_{n-k}$ on~$\C^E$ is given by~$[U_1,U_2,U_3]$, where~$U_{i} = [ u_{i,j} \,\, | \,\,  j \in [m_i]]$, where
\begin{align*}
  \begin{split}
    u_{1,j} &:= v_{e_j}  \,  \text{ for $j  \in [\tbinom{k}{2}]$ },\\
    u_{1,\tbinom{k}{2}+j} &:= \sum_{i=k+1}^n v_{\{j,i\}} \,\text{ for $j \in [k]$},\hspace{-50pt}\\
    u_{1,\tbinom{k}{2}+k+1} &:=  \sum_{k+1 \leq i < j \leq n} v_{\{i,j\}} ,
\end{split}
\begin{split}
u_{2,j} &:= v_{\{j,k+1\}}-v_{\{j,k+2\}} \,\text{ for $j \in [k]$},\\
 u_{2,k+1}&:= \sum_{i=k+3}^n (v_{\{k+1,i\}}-v_{\{k+2,i\} }), \\
\end{split}\\
u_{3,1}&:=  v_{\{k+1,k+2\}}-v_{\{k+1,k+3\}}  - v_{\{k+2,k+4\}}+v_{\{k+3,k+4\}}.\hspace{-100pt}
\end{align*}
\end{proposition}
\proof
For~$j \in [m_1]$, set~$V_{1,j} = \C u_{1,j}$, which is a~$1$-dimensional subspace of~$V$ on which $S_{n-k}$ acts trivially. Furthermore, for~$j \in[k]$, set
\begin{align*}
V_{2,j}:=  \left\{ \sum_{i=k+1}^n c_i v_{\{j,i\}} \,\, | \,\, \ c \in \C^n: c_i = 0 \text{ for } i \in [k], \, c\T \bm{1} = 0  \right\},
\end{align*}
and set
\begin{align*}
V_{2,k+1}:=  \left\{ \sum_{\substack{e = \{i,j\}: \\ k+1 \leq i < j \leq n}} (c_i+c_j) v_{\{i,j\}} \,\, | \,\, \ c \in \C^n: c_i = 0 \text{ for } i \in [k], \, c\T \bm{1} = 0  \right\} .
\end{align*}
Then~$V_{2,j}$ has dimension~$n-k-1$, for each~$j \in [k+1]$. Finally, define
\begin{align*}
 V_{3,1}=\left\{ \sum_{\substack{e = \{i,j\}: \\ k+1 \leq i < j \leq n}} \lambda_e v_e \,\, | \,\, \lambda_e \in \C, \, \forall m \in \{k+1,\ldots,n\}\,: \,\sum_{e: m \in e} \lambda_e =0  \right\},
\end{align*}
which is a subspace of~$V$ of dimension $\tbinom{n-k}{2}-(n-k) $. It is easy to verify that the~$V_{i,j}$ are pairwise orthogonal and~$S_{n-k}$-invariant.
\iffalse
To verify that~$V_{2,k+1}$ and~$V_{3,1}$ are pairwise orthogonal, observe that if~$c \in \C^n$ and~$\lambda_e \in \C$ satisfy the conditions, then 
$$
 \sum_{\substack{e = \{i,j\}: \\ k+1 \leq i < j \leq n}} (c_i+c_j) \lambda_e =\sum_{i\in \{k+1,\ldots,m\}} c_i \sum_{\substack{e \subseteq \{ k+1 ,\ldots, n\}: \\ i \in e}}\lambda_e = 0.
$$
\fi
 Moreover, $\dim V_{1,j} =1$ for~$j \in [m_1]$, $\dim V_{2,j} = n-k-1$ for~$j \in [m_2]$ and~$\dim V_{3,1} = \binom{n-k}{2}-(n-k)$.
We find that
\begin{align*}
\sum_{i=1}^3 \sum_{j=1}^{m_i} \dim V_{i,j} &= (\tbinom{k}{2}+k+1) \cdot 1 + (k+1)\cdot (n-k-1) + \tbinom{n-k}{2}-(n-k)= \tbinom{n}{2},
\end{align*}
so
\begin{align}\label{sumsum}
\C^E = \bigoplus_{i=1}^3 \bigoplus_{j=1}^{m_i} V_{i,j}.
\end{align}
To prove~\ref{3eig}, we count~$\text{dim}(\text{End}_{S_{n-k}}(\C^E)) = | (E \times E)/S_{n-k}|$.  We represent each equivalence class by its lexicographically smallest element. We find
{\small \begin{align*}
  \begin{split}
\tbinom{k}{2}^2 &\text{ orbits } (e_i,e_j) \text{ with } e_i,e_j \in \tbinom{[k]}{2},\\
\tbinom{k}{2}k &\text{ orbits } (e_i,\{j,k+1\}) \text{ with } e_i \in \tbinom{[k]}{2}, j \in [k],\\
\tbinom{k}{2}k &\text{ orbits } (\{j,k+1\},e_i) \text{ with } e_i \in \tbinom{[k]}{2}, j \in [k],\\
\tbinom{k}{2} &\text{ orbits } (e_i,\{k+1,k+2\}) \text{ with } e_i \in \tbinom{[k]}{2},\\
\tbinom{k}{2} &\text{ orbits } (\{k+1,k+2\},e_i) \text{ with } e_i \in \tbinom{[k]}{2},\\
k^2 &\text{ orbits } (\{i,k+1\},\{j,k+1\}) \text{ with } i,j \in [k],\\
k^2 &\text{ orbits } (\{i,k+1\},\{j,k+2\}) \text{ with } i,j \in [k],\\
\end{split}
\begin{split}
k &\text{ orbits } (\{i,k+1\},\{k+1,k+2\}) \text{ with } i \in [k],\\
k &\text{ orbits } (\{k+1,k+2\},\{i,k+1\}) \text{ with } i \in [k],\\
k &\text{ orbits } (\{i,k+1\},\{k+2,k+3\}) \text{ with } i \in [k],\\
k &\text{ orbits } (\{k+1,k+2\},\{i,k+3\}) \text{ with } i \in [k],\\
1 &\text{ orbit } (\{k+1,k+2\},\{k+1,k+2\}),\\
1 &\text{ orbit } (\{k+1,k+2\},\{k+1,k+3\}),\\
1 &\text{ orbit } (\{k+1,k+2\},\{k+3,k+4\}),
\end{split}
\end{align*}}
So~$|(E \times E) / S_{n-k}|=3+4k+2k^2+\tbinom{k}{2}(\tbinom{k}{2}+2k+2)$. A straightforward calculation shows that this number is equal to~$m_1^2+m_2^2+m_3^2 = (\tbinom{k}{2}+k+1)^2+(k+1)^2+1^2$, so~\ref{3eig} is satisfied. 

% For $\omega \in (E \times E)/S_{n-k}$, write~$a_{\omega}$ for the common value~$A_{(e_1,e_2)}$ in entries~$(e_1,e_2) \in \omega$. In what follows, we denote any~$\omega \in(E \times E)/S_{n-k}$ by its lexicographically smallest element. 

Observe that
 $\mu u_{1,1} \mapsto  \mu u_{1,j}$ gives an~$S_{n-k}$-isomorphism from~$V_{1,1}$ to~$V_{1,j}$ for~$j\in [m_1]$. Moreover,  $\sum_{i=k+1}^n c_i v_{\{1,i\}}\mapsto  \sum_{i=k+1}^n c_i v_{\{2,j\}}$ gives an~$S_{n-k}$-isomorphism~$V_{2,1} \to V_{2,j}$ for~$j\in[k]$  and
 $$
\sum_{i=k+1}^n c_i v_{\{1,i\}}\mapsto   \sum_{\substack{e = \{i,j\}: \\ k+1 \leq i < j \leq n}} (c_i+c_j) v_{\{i,j\}}
$$
 gives an~$S_{n-k}$-isomorphism~$V_{2,1} \to V_{2,k+1}$.  Together with~\ref{3eig}, this implies that the~$V_{i,j}$ indeed form a decomposition of~$V$ into \emph{irreducible} $S_{n-k}$-submodules (as any further decomposition, or representation, or equivalence among the~$V_{i,j}$ would imply that the squares of the multiplicities of the
irreducible representations of~$V$ sum to a number strictly larger than~$\sum_{i=1}^k m_i^2 \geq \dim (\text{End}_{S_{n-k}}(V))$, which contradicts the fact that~$\Phi$ from~\eqref{PhiC} is a linear bijection), and we have~$V_{i,j} = \C S_{n-k} \cdot u_{i,j}$. The given~$S_{n-k}$-isomorphisms thus imply that~\ref{2eig} is satisfied. By~\eqref{sumsum}, also~\ref{1eig} is satisfied. 
\endproof
We use the above proposition to prove Theorem~\ref{Antheorem}, as follows:
\begin{enumerate}
\item  First we use the representative set of Proposition~\ref{u1u2u3} to find a block-diagonalization of~$A^{(n)}$. This results in three blocks of constant size, with polynomials in~$n$ as entries.  
\item We divide rows and columns simultaneously by common expressions including~$n$, so that the three blocks each can be decomposed as a constant matrix (containing the coefficients before the highest power of~$n$ in the original block, up to a possible scaling by the division we did before) + another matrix whose entries are fractions with constant numerator and non-constant denominator (an expression in~$n$).

For large~$n$, the first (constant) matrix becomes  dominant. So for~$A^{(n)} \succeq 0$ for all~$n\geq k$, it is a necessary condition that the three constant matrices are positive semidefinite.

\item It will turn out that this condition also is sufficient. The matrices containing fractions with~$n$ in the denominator can be shown to be positive semidefinite for all~$n$ (assuming that the constant matrices are positive semidefinite).
   \end{enumerate}

\begin{proof}[Proof of Theorem \ref{Antheorem}]
Let~$(A^{(n)})_{n \geq k}$ be a series of matrices satisfying the conditions of the theorem. Write $a_{(e_i,e_j)}$ for the~$e_i,e_j$-th entry of~$A^{(k+4)}$, for~$e_i, e_j \in \tbinom{[k+4]}{2}$.
First consider~$n  \geq k+4$ and let again~$E:=\tbinom{[n]}{2}$. We use the representative set~$U_1,U_2,U_3$ for the action of~$S_{n-k}$ on~$\C^E$ from Proposition~\ref{u1u2u3}, and setting~$U_i^{(n)}:=U_i$ for~$i \in [3]$ to indicate the dependence on~$n$, we find that
\begin{align} \label{u1au1}
(U_1^{(n)})\T A^{(n)} U_1^{(n)} =\quad
\setlength\aboverulesep{1pt}\setlength\belowrulesep{1pt}
    \setlength\cmidrulewidth{0.5pt}
\begin{blockarray}{cccc}
 &   \tbinom{k}{2} &  k  & 1 \\
\begin{block}{c(c|c|c)}
   \tbinom{k}{2}   & C_{11}  & (n-k)C_{12}   & \tbinom{n-k}{2} C_{13}  \\\cmidrule{2-4}
 k&  (n-k)C_{12}\T   &   (n-k) C_{22}  & \tbinom{n-k}{2}C_{23} \\    \cmidrule{2-4}
 1&  \tbinom{n-k}{2}C_{13}\T  &  \tbinom{n-k}{2}C_{23}\T  &  \tbinom{n-k}{2}C_{33} \\
\end{block}
\end{blockarray} \,\,,
\end{align}
where 
\begin{alignat*}{3}
C_{11} &= \left( a_{(e_i,e_j)} \right)_{e_i,e_j \in \tbinom{[k]}{2}},  \quad \quad \,\,\,\,\,\,\, & C_{12} &= \left( a_{(e_i,\{j,k+1\})} \right)_{\substack{e_i \in \tbinom{[k]}{2} \\ j\in [k]  }}, \\ 
C_{13} &=  \left( a_{(e_i,\{k+1,k+2\})} \right)_{\substack{e_i \in \tbinom{[k]}{2}}},  &  C_{22} &= \left( a_{(\{i,k+1\},\{j,k+1\})} + (n-k-1) a_{(\{i,k+1\},\{j,k+2\})}  \right)_{i,j \in [k]}, \\ 
C_{23} &=  \mathrlap{\left( 2a_{\{j,k+1\},\{k+1,k+2\}} + (n-k-2) a_{\{j,k+1\},\{k+2,k+3\}}  \right)_{j \in [k]},} && \\
C_{33} &= \mathrlap{a_{(\{k+1,k+2\},\{k+1,k+2\})}  \hspace{-2pt} +\hspace{-2pt} 2(n-k-2)  a_{(\{k+1,k+2\},\{k+1,k+3\})} \hspace{-2pt} +\hspace{-2pt} \tbinom{n-k-2}{2} a_{(\{k+1,k+2\},\{k+3,k+4\})}.}&&\hspace{-2pt}
\end{alignat*}
Moreover, 
\begin{align} \label{secondmat}
\setlength\aboverulesep{1pt}\setlength\belowrulesep{1pt}
    \setlength\cmidrulewidth{0.5pt}
(U_2^{(n)})\T A^{(n)} U_2^{(n)}  = 
\begin{blockarray}{ccc}
 &  k&  1 \\
\begin{block}{c(c|c)}
   k& 2 D_{11} &   2(n-k-2)D_{12}  \\
     \cmidrule{2-3}
 1&    2(n-k-2)D_{12}\T \phantom{\binom{T}{T}}  &  2(n-k-2)D_{22}  \\  
\end{block}
\end{blockarray},
\end{align}
where 
\begin{align}
D_{11} &=   \left(a_{(\{i,k+1\},\{j,k+1\})}  -   a_{(\{i,k+1\},\{j,k+2\})}\right)_{i,j \in [k]},\notag \\
D_{12} &=  \left( a_{\{j,k+1\},\{k+1,k+2\}} - a_{\{j,k+1\},\{k+2,k+3\}}  \right)_{j \in [k]},\notag \\
D_{22} &=  a_{(\{k+1,k+2\},\{k+1,k+2\})} + (n-k-4)  a_{(\{k+1,k+2\},\{k+1,k+3\})}  \label{d22} \\&\quad\quad  -(n-k-3) a_{(\{k+1,k+2\},\{k+3,k+4\})}.\notag
\end{align}
Finally, 
\begin{align*}
(U_3^{(n)})\T A^{(n)} U_3^{(n)}  = 4  (a_{(\{k+1,k+2\},\{k+1,k+2\})} - 2  a_{(\{k+1,k+2\},\{k+1,k+3\})} +  a_{(\{k+1,k+2\},\{k+3,k+4\})}).
\end{align*} 
For brevity of notation, we set~$x:= a_{(\{k+1,k+2\},\{k+1,k+2\})}$, $y:= a_{(\{k+1,k+2\},\{k+1,k+3\})}$ and~$z:=  a_{(\{k+1,k+2\},\{k+3,k+4\})}$. Then~$(U_3^{(n)})\T A^{(n)} U_3^{(n)} \succeq 0$ if and only if~$x-2y+z \geq 0$.  

By~\eqref{secondmat}, the matrix~$(U_2^{(n)})\T A^{(n)} U_2^{(n)}$ is positive semidefinite if and only if the matrix
\begin{align} \label{secondmat2}
&\setlength\aboverulesep{0pt}\setlength\belowrulesep{0pt}
    \setlength\cmidrulewidth{0.5pt}
\begin{blockarray}{ccc}
 &   k&   1 \\
\begin{block}{c(c|c)}
   k&  D_{11} &   D_{12}  \\\cmidrule{2-3}
 1&   D_{12}\T    &  (n-k-2)^{-1}D_{22}  \\  
\end{block}
\end{blockarray} 
= \setlength\aboverulesep{1pt}\setlength\belowrulesep{1pt}
    \setlength\cmidrulewidth{0.5pt}\begin{blockarray}{ccc}
 &   k  &  1 \\
\begin{block}{c(c|c)}
   k& D_{11}    & D_{12}  \\\cmidrule{2-3}
 1&  D_{12}\T   &  y-z \\  
\end{block}
\end{blockarray} +(n-k-2)^{-1} \cdot \, \begin{blockarray}{ccc}
 &  k  &  1 \\
\begin{block}{c(c|c)}
   k& 0 &   0  \\\cmidrule{2-3}
 1&    0 &  x-2y+z\\  
\end{block}
\end{blockarray}   \, ,
\end{align}
is positive semidefinite. Here we used~\eqref{d22} to write~$D_{22} = (n-k-2)(y-z) +x -2y+z$. Both matrices in this last sum do not depend on~$n$, only the term~$(n-k-2)^{-1}$ depends on~$n$. Assuming~$(U_3^{(n)})\T A^{(n)} U_3^{(n)} \geq 0$, i.e.,~$x-2y+z \geq 0$, the matrix in~\eqref{secondmat2} is positive semidefinite for all~$n \geq k+4$ if and only if the first summand at the right hand side in~\eqref{secondmat2} is positive semidefinite (this is seen by letting~$n \to \infty$).

By~\eqref{u1au1}, the matrix~$(U_1^{(n)})\T A^{(n)} U_1^{(n)}$ is positive semidefinite if and only if the matrix
\begin{align} \label{u1au2continue}
\setlength\aboverulesep{1pt}\setlength\belowrulesep{1pt}
    \setlength\cmidrulewidth{0.5pt}
\begin{blockarray}{cccc}
 &  \tbinom{k}{2} &   k &  1 \\
\begin{block}{c(c|c|c)}
 \tbinom{k}{2}   &  C_{11} &  C_{12}  &   C_{13} \\\cmidrule{2-4}
 k&   C_{12}\T     &  (n-k)^{-1} C_{22}   & (n-k)^{-1} C_{23} \\    \cmidrule{2-4}
 1&  C_{13}\T    & (n-k)^{-1} C_{23}\T   & \tbinom{n-k}{2}^{-1}C_{33} \\
\end{block}
\end{blockarray} \,\,,
\end{align}
is positive semidefinite. To make the expressions independent of~$n$, it will turn out to be useful to use the relations
\begin{align*}
C_{22} &= (n-k)  \underbrace{(a_{(\{i,k+1\},\{j,k+2\})})_{i,j\in [k]}}_{=: \, C_{22}'} + \underbrace{(a_{(\{i,k+1\},\{j,k+1\})}-a_{(\{i,k+1\},\{j,k+2\})})_{i,j\in [k]}}_{=: \, C_{22}''}, \\
C_{23} &= (n-k)   \underbrace{(a_{(\{i,k+1\},\{k+2,k+3\})})_{i \in [k]}}_{=: \, C_{23}'}  + \underbrace{2(a_{(\{i,k+1\},\{k+1,k+2\})}-a_{(\{i,k+1\},\{k+2,k+3\})})_{i \in [k]}}_{=:\, C_{23}''}, \\
C_{33} &=  \tbinom{n-k}{2}  \underbrace{ a_{(\{k+1,k+2\},\{k+3,k+4\})}}_{=: \, C_{33}'} + 2(n-k-2) (a_{(\{k+1,k+2\},\{k+1,k+3\})} \hspace{-2pt} -\hspace{-2pt} a_{(\{k+1,k+2\},\{k+3,k+4\})})  
\\&\phantom{ = = } + (a_{(\{k+1,k+2\},\{k+1,k+2\})}-a_{(\{k+1,k+2\},\{k+3,k+4\})}),
\end{align*}
and we set~$C_{33}'' := C_{33} - \tbinom{n-k}{2} C_{33}'$. 

So the matrix in \eqref{u1au2continue} is equal to
\begin{align}
&\setlength\aboverulesep{0pt}\setlength\belowrulesep{0pt}
    \setlength\cmidrulewidth{0.5pt}
\begin{blockarray}{cccc}
 &   \tbinom{k}{2} & k& 1 \\
\begin{block}{c(c|c|c)}
 \tbinom{k}{2}    & C_{11}  & C_{12}   &  C_{13} 
  \\\cmidrule{2-4}
 k&   C_{12}\T     &   C_{22}'  & C_{23}' \\  \cmidrule{2-4}
 1&  C_{13}\T &  (C_{23}')\T  & C_{33}' \\
\end{block}
\end{blockarray} \hspace{-10pt} \phantom{=+} \setlength\aboverulesep{0pt}\setlength\belowrulesep{0pt}
    \setlength\cmidrulewidth{0.5pt}+(n-k)^{-1} \begin{blockarray}{cccc}
 &  \tbinom{k}{2} &   k &  1 \\
\begin{block}{c(c|c|c)}
 \tbinom{k}{2}   & 0 &  0  &  0 \\
  \cmidrule{2-4}
 k&  0   &     C_{22}''  & C_{23}'' \\  
\cmidrule{2-4}
 1&  0    & (C_{23}'')\T  & 2(n-k-1)^{-1}C_{33}'' \\ 
\end{block}
\end{blockarray}  \label{matrixvergelijking}
\end{align}
To decompose the matrix in the second summand, we notice that~$C_{33''} = 2(n-k-1)(y-z) + (x-2y+z)$. Also we have~$C_{22}'' = D_{11}$ and~$C_{23}'' = 2D_{12}$. Therefore:
\begin{align*}
\tfrac{1}{n-k} &\begin{pmatrix}
C_{22}'' & C_{23}''\\
 (C_{23}'')\T& 2 (n-k-1)^{-1} C_{33}''
\end{pmatrix} =\tfrac{1}{n-k}\begin{pmatrix}
D_{11} & 2D_{12}\\
 (2D_{12})\T& 4(y-z)
\end{pmatrix} +\tfrac{1}{\tbinom{n-k}{2}} \begin{pmatrix}
0 & 0\\
 0\T& x-2y+z
\end{pmatrix}
\end{align*}
Assuming~$(U_3^{(n)})\T A^{(n)} U_3^{(n)} \succeq 0$ and~$(U_2^{(n)})\T A^{(n)} U_2^{(n)} \succeq 0$ for all~$n \geq k+4$, this matrix is positive semidefinite. So, assuming~$(U_3^{(n)})\T A^{(n)} U_3^{(n)} \succeq 0$ and~$(U_2^{(n)})\T A^{(n)} U_2^{(n)} \succeq 0$ for all~$n \geq k+4$, the matrix in~\eqref{matrixvergelijking} (and hence~$(U_1^{(n)})\T A^{(n)} U_1^{(n)})$ is positive semidefinite for all~$n\geq k+4$ if and only if the matrix
\begin{align*}
&\setlength\aboverulesep{1pt}\setlength\belowrulesep{1pt}
    \setlength\cmidrulewidth{0.5pt}
\begin{blockarray}{cccc}
 &   \tbinom{k}{2} & k& 1 \\
\begin{block}{c(c|c|c)}
 \tbinom{k}{2}    & C_{11}  & C_{12}   &  C_{13} 
  \\\cmidrule{2-4}
 k&   C_{12}\T     &   C_{22}'  & C_{23}' \\  \cmidrule{2-4}
 1&  C_{13}\T &  (C_{23}')\T  & C_{33}' \\
\end{block}
\end{blockarray} 
\end{align*}
is positive semidefinite (this is again seen by letting~$n \to \infty$). Putting it all together, we find that
\begin{align}
A^{(n)} \succeq 0 \text{ for all $n \geq k+4$}\,\,\,\, \Longleftrightarrow\,\,\, \text{for all $n \geq k+4$}:(U_i^{(n)})\T A^{(n)} U_i^{(n)} \succeq 0 \,\,\forall i \in [3] \,\,\,\,\Longleftrightarrow \notag \\ 
\setlength\aboverulesep{1pt}\setlength\belowrulesep{1pt}
    \setlength\cmidrulewidth{0.5pt}
\begin{blockarray}{cccc}
 &   \tbinom{k}{2} & k& 1 \\
\begin{block}{c(c|c|c)}
 \tbinom{k}{2}    & C_{11}  & C_{12}   &  C_{13} 
  \\\cmidrule{2-4}
 k&   C_{12}\T     &   C_{22}'  & C_{23}' \\  \cmidrule{2-4}
 1&  C_{13}\T &  (C_{23}')\T  & C_{33}' \\
\end{block}
\end{blockarray}  \succeq 0, \begin{blockarray}{ccc}
 &   k&    1 \\
\begin{block}{c(c|c)}
   k& D_{11} &   D_{12}  \\
\cmidrule{2-3}
 1&   D_{12}\T   &  y-z \\  
\end{block}
\end{blockarray} \succeq 0, x-2y + z \geq 0. \label{threematrices2}
\end{align}
Since for~$n \in \N$ with~$k \leq n < k+4$, the matrix~$A^{(n)}$ is a principal submatrix of~$A^{(n')}$ (for any~$n' \geq k+4$) we even have that $A^{(n)} \succeq 0 \text{ for all $n \geq k$}$ if and only if the three matrices in~\eqref{threematrices2} are positive semidefinite.
 \end{proof}

\section{Proof of Propsition~\ref{drieprop}\label{smallpropproof}}
We continue by giving the proof of Proposition~\ref{drieprop}. It states that for all sequences of matrices coming from multigraphs as described in Section~\ref{howtouse}, the value in~\eqref{blockvalue} is nonnegative. Recall that this implies that~$A^{(n')}$ is positive semidefinite for all~$n' \geq k$ if and only if the two block matrices from Theorem~$\ref{Antheorem}$ are positive semidefinite (where ~$(A^{(n')})_{n' \geq k}$ denotes the sequence of matrices corresponding to~$\gamma$).
\begin{proof}[Proof of Proposition \ref{drieprop}]
 Recall that~$(A^{(n')})_{n' \geq k}$ is the sequence of matrices corresponding to~$\gamma$. We aim to prove that
$$
a_{(\{k+1,k+2\},\{k+1,k+2\})} - 2  a_{(\{k+1,k+2\},\{k+1,k+3\})} +  a_{(\{k+1,k+2\},\{k+3,k+4\})} \geq 0.
$$
In this case we have
\begin{align*}
a_{(\{k+1,k+2\},\{k+1,k+2\})} &= 2 b_{\gamma +2 v_{\{k+1,k+2\}}}, \\
a_{(\{k+1,k+2\},\{k+1,k+3\})} &= b_{\gamma +v_{\{k+1,k+2\}}+v_{\{k+1,k+3\}}}, \\
a_{(\{k+1,k+2\},\{k+3,k+4\})} &= b_{\gamma +v_{\{k+1,k+2\}}+v_{\{k+3,k+4\}}}, 
\end{align*}
For any sequence~$\underline{e} = (e_1\ldots,e_d) \in E^d$, we write~$\rho(\underline{e}) :=\prod_{i=1}^d \tfrac{1}{|e_1 \cup \ldots \cup e_i|}$. 
Let~$\underline{e} = (e_1\ldots,e_d)$ be any sequence with~$\alpha_n(\underline{e}) = \gamma + 2 v_{\{k+1,k+2\}}$.  Let~$j_1,j_2 \in [d]$ be the indices with~$e_{j_1} = \{k+1,k+2\}$ and~$e_{j_2} = \{k+1,k+2\}$ .  Let~$\underline{e}'_1$ and~$\underline{e}''_1$ be the sequences obtained from~$\underline{e}$ by replacing~$e_{j_1}$ with~$\{k+1,k+3\}$ and with~$\{k+3,k+4\}$, respectively. Similarly, let~$\underline{e}'_2$ and~$\underline{e}''_2$ be the sequences obtained from~$\underline{e}$ by replacing~$e_{j_2}$ with~$\{k+1,k+3\}$ and with~$\{k+3,k+4\}$, respectively. 
Then
\begin{align}
&a_{(\{k+1,k+2\},\{k+1,k+2\})} - 2  a_{(\{k+1,k+2\},\{k+1,k+3\})} +  a_{(\{k+1,k+2\},\{k+3,k+4\})} \notag
\\&=   2 b_{\gamma +2 v_{\{k+1,k+2\}}} - 2 b_{\gamma +v_{\{k+1,k+2\}}+v_{\{k+1,k+3\}}} +  b_{\gamma +v_{\{k+1,k+2\}}+v_{\{k+3,k+4\}}}     \notag\\  
&=  2 \sum_{\substack{\underline{e} = (e_1\ldots,e_d) \in E^d \\ \alpha_n(\underline{e}) = \gamma + 2 v_{\{k+1,k+2\}}}} \hspace{-10pt}\rho(\underline{e}) -2 \sum_{\substack{\underline{e} = (e_1\ldots,e_d) \in E^d \\ \alpha_n(\underline{e}) = \gamma +  v_{\{k+1,k+2\}} + v_{\{k+1,k+3\}}}} \rho(\underline{e})+ \sum_{\substack{\underline{e} = (e_1\ldots,e_d) \in E^d \\ \alpha_n(\underline{e}) = \gamma +  v_{\{k+1,k+2\}} + v_{\{k+3,k+4\}}}} \hspace{-8pt} \rho(\underline{e}) \notag 
\\&= \sum_{\substack{\underline{e} = (e_1\ldots,e_d) \in E^d \\ \alpha_n(\underline{e}) = \gamma + 2 v_{\{k+1,k+2\}}}} \left( 2\rho(\underline{e}) -2 (\rho(\underline{e}'_1) +\rho(\underline{e}'_2) ) +   \rho(\underline{e}''_1) + \rho(\underline{e}''_2) \right)   \notag 
\\&= 2 \sum_{\substack{\underline{e} = (e_1\ldots,e_d) \in E^d \\ \alpha_n(\underline{e}) = \gamma + 2 v_{\{k+1,k+2\}}}} \left( \rho(\underline{e}) -2 \rho(\underline{e}'_1)  +   \rho(\underline{e}''_1) \right)     \label{termsom}
\end{align}
Now we consider any fixed term in this sum, corresponding to a sequence~$\underline{e} = (e_1\ldots,e_d)$  with~$\alpha_n(\underline{e}) = \gamma + 2 v_{\{k+1,k+2\}}$. 
For~$i \in [d]$, write $\beta_i := \tfrac{1}{|e_1 \cup \ldots \cup e_i|}$. Let~$j:= \max \{j_1,j_2\}$. Then
\begin{align*}
\rho(\underline{e}) &=  \prod_{i=1}^d \tfrac{1}{\beta_i}, \quad \rho(\underline{e}'_1) =\rho(\underline{e}'_2) =  \prod_{i=1}^{j-1}  \tfrac{1}{\beta_i}\prod_{i=j}^{d}  \tfrac{1}{\beta_i+1}, \quad
\rho(\underline{e}''_1) =\rho(\underline{e}''_2)=  \prod_{i=1}^{j-1}  \tfrac{1}{\beta_i}\prod_{i=j}^{d}  \tfrac{1}{\beta_i+2}, \\ 
\end{align*}
where the equalities for~$ \rho(\underline{e}'_1)$ and~$ \rho(\underline{e}''_1) $ follow since~$e_i \subseteq [k]$ for each~$i \neq j_1,j_2$.  We compute
\begin{align*}
\rho(\underline{e}) -2 \rho(\underline{e}'_1)  +   \rho(\underline{e}''_1) &=
\prod_{i=1}^d \tfrac{1}{\beta_i} - 2 \prod_{i=1}^{j-1} \tfrac{1}{\beta_i}\prod_{i=j}^d \tfrac{1}{(\beta_i + 1) }+\prod_{i=1}^{j-1} \tfrac{1}{\beta_i} \prod_{i=j}^d \tfrac{1}{(\beta_i + 2) }
\\ &= \prod_{i=1}^{j-1} \tfrac{1}{\beta_i} \left(  \prod_{i=j}^d \tfrac{1}{\beta_i} - 2 \prod_{i=j}^d \tfrac{1}{(\beta_i + 1 ) }+ \prod_{i=j}^d \tfrac{1}{(\beta_i + 2) } \right) 
\\ & \geq  \prod_{i=1}^{j-1} \tfrac{1}{\beta_i} \left(   2 \sqrt{\prod_{i=j}^d \tfrac{1}{\beta_i} \prod_{i=j}^d \tfrac{1}{(\beta_i + 2) }}  - 2 \prod_{i=j}^d \tfrac{1}{(\beta_i + 1 ) }      \right) 
\\ & \geq  \prod_{i=1}^{j-1} \tfrac{1}{\beta_i} \left(   2 \sqrt{\prod_{i=j}^d \tfrac{1}{(\beta_i + 1)^2 }}  - 2 \prod_{i=j}^d \tfrac{1}{(\beta_i + 1 ) }      \right) 
= 0,
\end{align*}
where the first inequality follows from the inequality of arithmetic-geometric means. So every term in the sum in~\eqref{termsom} is nonnegative. This proves the proposition.
\end{proof}

\section{Computer verification of Theorem~\ref{theorem2}\label{computer}}
We report on the computational results. Recall that to verify Theorem~\ref{theorem2}, we perform the steps of Procedure~\ref{procver}, for each~$ d \leq 9$. We repeat these steps with comments about the computations and present a table with computational results.
\begin{enumerate}[label=(\roman*)]
\item We enumerate (up to isomorphism)  all multigraphs with exactly~$(d-2)$ edges without isolated vertices. To do this we used the programs \texttt{geng}  and \texttt{multig} from~\cite{dreadnaut}. 
\item\label{sequenced2} For each multigraph~$\underline{e}= (e_1,\ldots,e_{d-2})$ from the previous step, we set~$k:= |\cup_{i=1}^{d-2} e_i|  \leq 2(d-2)$. We relabel the vertices such that~$e_i \subseteq [k]$ for each~$i \in [d-2]$.
\item\label{sequenced3} For each~$n \geq k$, set~$\gamma(n) := \alpha_n(\underline{e})\in \N_{d-2}^{\binom{n}{2}}$ and define~$A^{(n)} := Q_{\gamma(n)}$. We verify that~$A^{(n)}$ is positive semidefinite for each~$n \geq k$ by verifying that the three  block matrices given in Theorem~\ref{Antheorem} (which are constructed from the matrix~$A^{(k+4)}$) are positive semidefinite.

In the computations for Table~\ref{overview} below  we used each time  the sequence $A^{(n)} :=\gamma(n)! Q_{\gamma(n)}$ instead of $A^{(n)} := Q_{\gamma(n)}$, to cancel the factor $1/\gamma(n)!$ in the definition of~$Q_{\gamma(n)}$ (cf.~\eqref{qgammadef2}). This amounts to multiplying the sequence by a constant positive integer, as~$\gamma(n)!$ does not depend on~$n$, only on the edges in the multigraph~$\underline{e}$. 
\end{enumerate}
After performing these steps, we know that $Q_{\gamma(n)}$ is positive semidefinite for each~$n \geq k$ and for each multigraph, hence~$H(f_d) \succeq 0$, hence~$f_d$ is convex (for the considered fixed~$d$, but for all~$n$). 

In Table~\ref{overview} we show for fixed~$d$ the number of multigraphs with exactly~$d-2$ edges (up to isomorphism). Also we show the minimum of all eigenvalues obtained after computing for each of these multigraphs, for the sequence of matrices~$(A^{(n)})_{n\geq k}$ as defined in~\ref{sequenced3}, the eigenvalues of the three block matrices from Theorem~\ref{Antheorem} (which are constructed from the matrix~$A^{(k+4)}$).

\begin{table}[ht]
\begin{center}
    \begin{tabular}{| r | r ||  >{\bfseries} r |}
    \hline
    $d$ & multigraphs & $\lambda_{\text{min}}$   \\\hline 
 3 & 1 & 0.00357563 \\
4 & 3 & 0.00059703 \\
5 & 8 & 0.00015202 \\
6 & 23 &0.00004653 \\ 
7 &  66& 0.00001583 \\
8 & 212 & 0.00000576 \\
9 & 686 & 0.00000220 \\\hline                      
    \end{tabular}
\end{center}
  \caption{\small For each degree~$d$, we display the number of multigraphs  with~$d-2$ edges and the minimum of all eigenvalues found with the procedure explained above. \label{overview}}
\end{table} 
For a fixed degree~$d$, the largest~$k$ occurring in~\ref{sequenced2} is~$k=2(d-2)$. The largest block matrix from Theorem~\ref{Antheorem} has order~$\tbinom{k}{2}+k+1$. To also verify the theorem for~$d=10$, neither the computation of the eigenvalues of such a matrix, nor the enumeration of the multigraphs (for~$d=10$ there are $2389$ multigraphs on at most~$d-2=8$ edges) were limiting factors. The limiting factor was the computation of the entries~$\widehat{b}_{\gamma(k+4)+v_i+v_j}$ of the matrix~$\gamma(k+4)!Q_{\gamma(k+4)}$ (cf.~\eqref{qgammadef2}), in combination with the fact that one must compute for each multigraph many of these entries. It is plausible that with more computation time and better optimizations, also convexity of~$f_d$ for~$d=10$ can be verified. Since such a verification does not give more insight to prove the theorem for all~$d$ and the method is now clearly illustrated, we stopped at~$d=9$. 

\subsection*{Experimental observations}
For each~$d \leq 9$, the minimum of all eigenvalues found with the procedure described in the previous section is attained for the multigraph consisting of disjoint edges, as our computer experiments demonstrated (with the sequence $A^{(n)} :=\gamma(n)! Q_{\gamma(n)}$). It would be interesting to know if this holds for all~$d$. If this is proved to be true, it is not necessary to check for each~$d$ all nonisomorphic multigraphs on~$d-2$ edges anymore, as then only the single multigraph consisting of $d-2$ disjoint edges must be checked. 

During the experiments it was observed that for sequences~$(A^{(n)})_{n \geq k}$ coming from multigraphs as in step~\ref{sequenced3} above, the minimum eigenvalue seems to be attained at the matrix block of size $\tbinom{k}{2}+k+1$. Perhaps such a monotonicity property can be proved in general for sequences~$(A^{(n)})_{n \geq k}$ coming from a multigraph. (This block does not have the minimum eigenvalue for general sequences $(A^{(n)})_{n \geq k}$. This can be seen from the example in Remark~\ref{k0remark} with e.g.,~$(x,y,z)=(6,4,3)$.)

Finally, note that the upper left principal submatrix of size~$\tbinom{k}{2}$ of the  block of size~$\tbinom{k}{2}+k+1$  from Theorem~\ref{Antheorem} does not have symmetry in general: it is heavily dependent on the chosen multigraph and the rows and columns of this part are indexed by the edges of the chosen multigraph. We did not apply any symmetry reduction to this part. For a given multigraph, the automorphism group of the multigraph may be used to further reduce the problem. While this may give  a reduction for fixed multigraphs, it does not appear to remove the dependency on~$d$ in general. A new idea might be required to prove that~$f_d$ is convex for all~$d$.

\section{Possible extensions\label{extensions}}

It might be possible to generalize Theorem~\ref{Antheorem} to edge sizes~$L >2$. A representative set can be derived using the representation theory of the symmetric group (for~$k=0$, one can use that~$\C^{E}$ as an~$S_n$-module is in this case isomorphic to the permutation module~$M^{(n-L,L)}$, where the notation is as in Sagan~\cite{sagan}). To prove that the whole sequence of matrices is positive semidefinite, a necessary condition of positive semidefiniteness of certain constant matrices can probably be derived by considering the coefficients before the highest powers of~$n$ in the matrices. However, it is not immediately clear how to prove sufficiency,  and we leave this to further research. The present paper focuses on the ``power-of-two-model'', the case~$L=2$, which was the case Cardinaels, Borst and Van Leeuwaarden were primarily interested in (in the light-traffic regime). 

The current check of positive semidefiniteness of the constant matrices is numerical and done in matlab.\footnote{See the supplementary material attached to the arXiv submission.} Since the constant matrices are relatively small and their eigenvalues are considerably larger than zero, we concluded that they are positive definite. As suggested by an anonymous referee, explicit rational certificates of positive (semi)definiteness can be given. We tried this for small~$d$, but the  matlab code to generate the constant matrices exactly (using rational computations) is slow.  The minimum entry on the diagonal of~$D$ in an~$LDL\T$-decomposition of all constant matrices considered for~$d=3$ is $\tfrac{13}{2360}$, for $d=4$ it is $\tfrac{17}{21232}$, and for~$d=5$ it is $\tfrac{2341}{12369056}$. 

An anonymous referee noted that whereas deciding convexity in general is NP-hard, it follows from an article by G\"orlach et al.\ \cite{gorlach} that checking for convexity for symmetric polynomials of fixed degree can be done in a time polynomial in the number of variables. In this paper, the convexity check is done in time exponential in~$d$, but fully independent of~$n$. The method presented here can also be used for other (classes of) polynomials in~$\tbinom{n}{2}$ variables, given that the Hessian can be decomposed as a sum of matrices invariant under~$S_{n-k}$ for some fixed~$k$. The present work only focuses on the particular question about~$f_d$.

Other possible extensions follow from the experimental observations. To prove the conjecture for all~$d$, it may help to first show that the matrix~$B_1$ corresponding to the matching (the multigraph consisting of~$d-2$ disjoint edges) has the smallest eigenvalue. Then only this matrix must be checked, which possibly opens a way to a full proof of the conjecture.

\setcounter{equation}{0}
\setcounter{figure}{0}
\setcounter{theorem}{0}
\setcounter{section}{0}
\appendix
\section{Appendix}
\subsection{Explicit matrices for the case \texorpdfstring{$d= 3$}{d =3}}\label{d3verify}
First we consider~$d=3$. There is only one multigraph consisting of $d-2=1$ edge, namely the multigraph~$\underline{e}= (\{1,2\})$. Set~$k=2$, and define for~$n \in \N$ the sequence~$\gamma(n) := \alpha_n(\underline{e})$. We write down the value and the two block matrices obtained from~$A^{(k+4)} = A^{(6)}=1!Q_{\gamma(6)}$ from Theorem~\ref{Antheorem}.  By Example~\ref{examplegamma}, we have
$$
A^{(6)}_{i,j} = \frac{1}{|e_1 \cup i \cup j|}\left( \frac{1}{|e_1\cup i|} + \frac{1}{|e_1\cup j|} + \frac{1}{|i\cup j |}\right), \quad \text{ for } i,j \in \tbinom{[6]}{2}.
$$
The block matrices from Theorem~\ref{Antheorem} are constructed from this matrix.
The value of~\eqref{blockvalue} is~$\tfrac{1}{24} \geq 0$, and the two matrices are
\begin{align*} B_1 :=
\begin{pmatrix} 
\tfrac{   3}{4} & \tfrac{ 7}{18} & \tfrac{ 7}{18} & \tfrac{1}{4} \\
\tfrac{ 7}{18} & \tfrac{1}{4} & \tfrac{11}{48} & \tfrac{1}{6} \\
\tfrac{ 7}{18} & \tfrac{11}{48} & \tfrac{1}{4} & \tfrac{1}{6} \\
\tfrac{  1}{4} & \tfrac{  1}{6} & \tfrac{  1}{6} & \tfrac{1}{8} \\
\end{pmatrix}, \,\,\, B_2 :=\begin{pmatrix} 
\tfrac{   5}{36} & \tfrac{5}{48} & \tfrac{1}{16} \\
\tfrac{ 5}{48} & \tfrac{  5}{36} & \tfrac{1}{16} \\
\tfrac{   1}{16} & \tfrac{  1}{16} & \tfrac{  1}{24} \\
\end{pmatrix}.
\end{align*}
Both matrices are positive (semi)definite: we have~$\lambda_{\text{min}}(B_1)\approx 0.00357563$ and $\lambda_{\text{min}}(B_2)\approx 0.00837652$. 

\subsection{Explicit matrices for the case \texorpdfstring{$d=4$}{d =4}}
Next, we consider~$d=4$. There are 3 distinct multigraphs  up to isomorphism consisting of $d-2=2$ edges, namely the multigraphs $\underline{e}=(e_1,e_2)$ where~$e_1=\{1,2\}$ and~$e_2=\{1,2\}$, $\{1,3\}$ or $\{3,4\}$, respectively.

\subsubsection{The case \texorpdfstring{$e_1 = \{1,2\}$, $e_2 = \{1,2\}$}{e1 = e2 = 12}}

Let~$k=2$, and let~$e_1=\{1,2\}$ and~$e_2=\{1,2\}$, and~$\underline{e}=(e_1,e_2)$. Define for~$n \in \N$ the sequence~$\gamma(n) := \alpha_n(\underline{e})$. We compute~$A^{(k+4)} = 2! Q_{\gamma(k+4)}$ and then verify that the constant matrices from Theorem~\ref{Antheorem} are positive semidefinite with the computer. The value of~\eqref{blockvalue} is~$\tfrac{1}{24} \geq 0$, and the two matrices are
\begin{align*}
B_1 := 2\begin{pmatrix} 
\tfrac{   3}{4} & \tfrac{ 23}{72} & \tfrac{ 23}{72} & \tfrac{3}{16} \\
\tfrac{ 23}{72} & \tfrac{23}{144} & \tfrac{89}{576} & \tfrac{7}{72} \\
\tfrac{ 23}{72} & \tfrac{89}{576} & \tfrac{23}{144} & \tfrac{7}{72} \\
\tfrac{  3}{16} & \tfrac{  7}{72} & \tfrac{  7}{72} & \tfrac{1}{16} \\
\end{pmatrix}, \,\,\,\, B_2 := 2\begin{pmatrix} 
\tfrac{   43}{432} & \tfrac{149}{1728} & \tfrac{23}{576} \\
\tfrac{ 149}{1728} & \tfrac{  43}{432} & \tfrac{23}{576} \\
\tfrac{   23}{576} & \tfrac{  23}{576} & \tfrac{  1}{48} \\
\end{pmatrix}.
\end{align*}
Both matrices are positive (semi)definite: we have~$\lambda_{\text{min}}(B_1) \approx 0.00192960$ and $\lambda_{\text{min}}(B_2)\approx 0.00670304$.

\subsubsection{The case \texorpdfstring{$e_1 = \{1,2\}$, $e_2 = \{1,3\}$}{e1 = 12, e2=13}}
Let~$k=3$, and let~$e_1=\{1,2\}$ and~$e_2=\{1,3\}$, and~$\underline{e}=(e_1,e_2)$.  Define for~$n \in \N$ the sequence~$\gamma(n) := \alpha_n(\underline{e})$. We compute~$A^{(k+4)} = Q_{\gamma(k+4)}$ and then verify that the constant matrices from Theorem~\ref{Antheorem} are positive semidefinite with the computer. The value of~\eqref{blockvalue} is~$1/36 \geq 0$, and the two matrices are
\begin{align*}B_1 := 
\begin{pmatrix} 
  \tfrac{23}{36} & \tfrac{ 14}{27} & \tfrac{13}{27} & \tfrac{23}{72} & \tfrac{89}{288} & \tfrac{79}{288} & \tfrac{ 7}{36} \\
  \tfrac{14}{27} & \tfrac{ 23}{36} & \tfrac{13}{27} & \tfrac{23}{72} & \tfrac{79}{288} & \tfrac{89}{288} & \tfrac{ 7}{36} \\
  \tfrac{13}{27} & \tfrac{ 13}{27} & \tfrac{13}{27} & \tfrac{25}{96} & \tfrac{ 25}{96} & \tfrac{ 25}{96} & \tfrac{  1}{6} \\
  \tfrac{23}{72} & \tfrac{ 23}{72} & \tfrac{25}{96} & \tfrac{  1}{5} & \tfrac{  7}{40} & \tfrac{  7}{40} & \tfrac{  1}{8} \\
 \tfrac{89}{288} & \tfrac{79}{288} & \tfrac{25}{96} & \tfrac{ 7}{40} & \tfrac{  7}{40} & \tfrac{19}{120} & \tfrac{11}{96} \\
 \tfrac{79}{288} & \tfrac{89}{288} & \tfrac{25}{96} & \tfrac{ 7}{40} & \tfrac{19}{120} & \tfrac{  7}{40} & \tfrac{11}{96} \\
   \tfrac{7}{36} & \tfrac{  7}{36} & \tfrac{  1}{6} & \tfrac{  1}{8} & \tfrac{ 11}{96} & \tfrac{ 11}{96} & \tfrac{ 1}{12} \\
\end{pmatrix}, \,\,\, B_2 := \begin{pmatrix} 
 \tfrac{ 43}{360} & \tfrac{  41}{480} & \tfrac{  41}{480} & \tfrac{ 1}{20} \\
 \tfrac{ 41}{480} & \tfrac{143}{1440} & \tfrac{  17}{240} & \tfrac{7}{160} \\
 \tfrac{ 41}{480} & \tfrac{  17}{240} & \tfrac{143}{1440} & \tfrac{7}{160} \\
 \tfrac{   1}{20} & \tfrac{   7}{160} & \tfrac{   7}{160} & \tfrac{ 1}{36} \\
\end{pmatrix}.
\end{align*}
Both matrices are positive (semi)definite: we have~$\lambda_{\text{min}}(B_1) \approx 0.00101380$ and $\lambda_{\text{min}}(B_2)\approx  0.00384022$.

\subsubsection{The case \texorpdfstring{$e_1 = \{1,2\}$, $e_2 = \{3,4\}$}{e1 = 12, e2 = 34}}
Let~$k=4$, and let~$e_1=\{1,2\}$ and~$e_2=\{3,4\}$, and~$\underline{e}=(e_1,e_2)$. Define for~$n \in \N$ the sequence~$\gamma(n) := \alpha_n(\underline{e})$. We compute~$A^{(k+4)} = Q_{\gamma(k+4)}$ and then verify that the  constant matrices from Theorem~\ref{Antheorem} are positive semidefinite with the computer. The value of~\eqref{blockvalue} is~$\tfrac{1}{48} \geq 0$, and the two matrices are
\setcounter{MaxMatrixCols}{20}
{\footnotesize \begin{align*}
B_1 := \begin{pmatrix} 
 \tfrac{    3}{8} & \tfrac{79}{288} & \tfrac{79}{288} & \tfrac{79}{288} & \tfrac{79}{288} & \tfrac{   1}{4} & \tfrac{  7}{36} & \tfrac{  7}{36} & \tfrac{   1}{6} & \tfrac{   1}{6} & \tfrac{  1}{8} \\
 \tfrac{ 79}{288} & \tfrac{89}{288} & \tfrac{ 25}{96} & \tfrac{ 25}{96} & \tfrac{ 11}{48} & \tfrac{79}{288} & \tfrac{  7}{40} & \tfrac{19}{120} & \tfrac{  7}{40} & \tfrac{19}{120} & \tfrac{11}{96} \\
 \tfrac{ 79}{288} & \tfrac{ 25}{96} & \tfrac{89}{288} & \tfrac{ 11}{48} & \tfrac{ 25}{96} & \tfrac{79}{288} & \tfrac{  7}{40} & \tfrac{19}{120} & \tfrac{19}{120} & \tfrac{  7}{40} & \tfrac{11}{96} \\
 \tfrac{ 79}{288} & \tfrac{ 25}{96} & \tfrac{ 11}{48} & \tfrac{89}{288} & \tfrac{ 25}{96} & \tfrac{79}{288} & \tfrac{19}{120} & \tfrac{  7}{40} & \tfrac{  7}{40} & \tfrac{19}{120} & \tfrac{11}{96} \\
 \tfrac{ 79}{288} & \tfrac{ 11}{48} & \tfrac{ 25}{96} & \tfrac{ 25}{96} & \tfrac{89}{288} & \tfrac{79}{288} & \tfrac{19}{120} & \tfrac{  7}{40} & \tfrac{19}{120} & \tfrac{  7}{40} & \tfrac{11}{96} \\
 \tfrac{    1}{4} & \tfrac{79}{288} & \tfrac{79}{288} & \tfrac{79}{288} & \tfrac{79}{288} & \tfrac{   3}{8} & \tfrac{   1}{6} & \tfrac{   1}{6} & \tfrac{  7}{36} & \tfrac{  7}{36} & \tfrac{  1}{8} \\
 \tfrac{   7}{36} & \tfrac{  7}{40} & \tfrac{  7}{40} & \tfrac{19}{120} & \tfrac{19}{120} & \tfrac{   1}{6} & \tfrac{   1}{8} & \tfrac{ 11}{96} & \tfrac{   1}{9} & \tfrac{   1}{9} & \tfrac{ 1}{12} \\
 \tfrac{   7}{36} & \tfrac{19}{120} & \tfrac{19}{120} & \tfrac{  7}{40} & \tfrac{  7}{40} & \tfrac{   1}{6} & \tfrac{ 11}{96} & \tfrac{   1}{8} & \tfrac{   1}{9} & \tfrac{   1}{9} & \tfrac{ 1}{12} \\
 \tfrac{    1}{6} & \tfrac{  7}{40} & \tfrac{19}{120} & \tfrac{  7}{40} & \tfrac{19}{120} & \tfrac{  7}{36} & \tfrac{   1}{9} & \tfrac{   1}{9} & \tfrac{   1}{8} & \tfrac{ 11}{96} & \tfrac{ 1}{12} \\
 \tfrac{    1}{6} & \tfrac{19}{120} & \tfrac{  7}{40} & \tfrac{19}{120} & \tfrac{  7}{40} & \tfrac{  7}{36} & \tfrac{   1}{9} & \tfrac{   1}{9} & \tfrac{ 11}{96} & \tfrac{   1}{8} & \tfrac{ 1}{12} \\
 \tfrac{    1}{8} & \tfrac{ 11}{96} & \tfrac{ 11}{96} & \tfrac{ 11}{96} & \tfrac{ 11}{96} & \tfrac{   1}{8} & \tfrac{  1}{12} & \tfrac{  1}{12} & \tfrac{  1}{12} & \tfrac{  1}{12} & \tfrac{ 1}{16} \\
\end{pmatrix}, \,\,B_2 := \begin{pmatrix} 
 \tfrac{   5}{72} & \tfrac{  5}{96} & \tfrac{17}{360} & \tfrac{17}{360} & \tfrac{1}{32} \\
 \tfrac{   5}{96} & \tfrac{  5}{72} & \tfrac{17}{360} & \tfrac{17}{360} & \tfrac{1}{32} \\
 \tfrac{ 17}{360} & \tfrac{17}{360} & \tfrac{  5}{72} & \tfrac{  5}{96} & \tfrac{1}{32} \\
 \tfrac{ 17}{360} & \tfrac{17}{360} & \tfrac{  5}{96} & \tfrac{  5}{72} & \tfrac{1}{32} \\
 \tfrac{   1}{32} & \tfrac{  1}{32} & \tfrac{  1}{32} & \tfrac{  1}{32} & \tfrac{1}{48} \\
\end{pmatrix}.
\end{align*}}Both matrices are positive (semi)definite: we have~$\lambda_{\text{min}}(B_1) \approx 0.00059703$ and $\lambda_{\text{min}}(B_2)\approx  0.00253196$.

\section*{Acknowledgements}
 The author thanks Monique Laurent for explaining this problem to him, for suggesting that a symmetry reduction can be applied, and for valuable discussions and comments. Also the author wants to thank Daniel Brosch and Andries Steenkamp for useful discussions.

Moreover, the author thanks the anonymous referees and the associate editor for their careful reading and good suggestions to improve the content and the presentation of this paper.

\end{document}